\documentclass[final]{siamltex}

\usepackage{amsfonts}
\usepackage{subfigure}

\title{Convergence analysis of sectional methods for solving aggregation population balance equations: The fixed pivot technique}

\author{Ankik Kumar Giri \footnotemark[2]\ \footnotemark[3]
\and Erika Hausenblas \footnotemark[2]}

\begin{document}
\maketitle

\renewcommand{\thefootnote}{\fnsymbol{footnote}}

\footnotetext[2]{Institute for Applied Mathematics, Montan University Leoben, Franz Josef Stra{\ss}e 18, A-8700 Leoben, Austria({\tt ankik-kumar.giri@unileoben.ac.at})}
\footnotetext[3]{Institute for Analysis and Numerics, Otto-von-Guericke University Magdeburg, Universit\"{a}tsplatz 2, D-39106 Magdeburg, Germany}

\renewcommand{\thefootnote}{\arabic{footnote}}

%

\begin{abstract}
In this paper, we introduce the
convergence analysis of the fixed pivot technique given by S.
Kumar and Ramkrishna \cite{Kumar:1996-1} for the
 nonlinear aggregation population balance equations which are of substantial interest
  in many areas of science: colloid chemistry, aerosol physics, astrophysics, polymer
  science, oil recovery dynamics, and mathematical biology. In particular,
   we investigate the convergence for five different types of uniform and non-uniform
   meshes which turns out that the fixed pivot technique is second order convergent on
    a uniform and non-uniform smooth meshes. Moreover, it yields first order convergence
    on a locally uniform mesh. Finally, the analysis exhibits that the method does not
    converge on an oscillatory and non-uniform random meshes. Mathematical results of
    the convergence analysis are also demonstrated numerically.
\end{abstract}



\pagestyle{myheadings}
\thispagestyle{plain}
\markboth{ANKIK KUMAR GIRI AND ERIKA HAUSENBLAS}{CONVERGENCE OF THE FIXED PIVOT TECHNIQUE FOR AGGREGATION PBEs}

\section{Introduction}
The continuous aggregation population balance equation (PBE) or Smoluchowski coagulation equation describes the kinetics of particle growth in which particles can aggregate via binary interaction to form larger particles. This model arises in many fields of science and engineering: kinetics of phase transformations in binary alloys such as segregation of binary alloys, aggregation of red blood cells in biology, fluidized bed granulation processes, aerosol physics, i.e\ the evolution of a system of solid or liquid particles suspended in a gas, formation of planets in astrophysics, polymer science and many more. The nonlinear continuous aggregation population balance equation is given by
\begin{eqnarray}\label{pbe}
\frac{\partial f(t,x)}{\partial t}  =  \frac{1}{2}\int_{0}^{x} K (x-y,y)&&f(t,x-y)f(t,y)dy\\
&&\nonumber - \int_{0}^{\infty} K(x,y)f(t,x)f(t,y)dy,
\end{eqnarray}
with
\vspace{-0.80cm}
\begin{eqnarray*}
f(x,0) = f^{\mbox{in}}(x)
\; \geq \; 0, \hspace{0.2in} x \in ]0,\infty[.
\end{eqnarray*}
where the variables $x> 0$ and $t\geq 0$ denote the size of the particles and time respectively. The number density of particles of size $x$ at time $t$ is denoted by $f(x,t)\geq 0$. The aggregation kernel $K(x,y)\geq 0$ represents the rate at which particles of size $x$ coalesce with those of size $y$. It will be assumed throughout that $K(x,y)=K(y,x)$ for all $x$, $y > 0$, i.e.\ symmetric and $K(x,y)=0$ for either $x=0$ or $y=0$.\\
  Mathematical results on existence and uniqueness of solutions to equation (\ref{pbe}) can be found in \cite{Dubovskii:1994, Dubovskii:1996, Escobedo:2003, GIRI:2010EXT, Lamb:2004, Laurencot:2000, Laurencot:2002DC, Mclaughlin:1997I, Mclaughlin:1997, Stewart:1990} for different classes of aggregation kernels. The pure aggregation PBE (\ref{pbe}) can be solved analytically only for some specific examples of aggregation kernels, see \cite{Dubovskii:1992, Fournier:2005, Fournier:2006}. In general we need to solve them numerically. To apply a numerical method, first we need to consider the following truncated form of the problem (\ref{pbe}) by taking a finite computational domain.
\begin{eqnarray}\label{trunc pbe fpt}
\frac{\partial n(t,x)}{\partial t}  =  \frac{1}{2}\int_{0}^{x} K(x-y,y)&&n(t,x-y)n(t,y)dy\\
&&\nonumber- \int_{0}^{x_{\mbox{max}}} K(x,y)n(t,x)n(t,y)dy,
\end{eqnarray}
with
\vspace{-0.80cm}
\begin{eqnarray*}
n(x,0) = n^{\mbox{in}}(x)\geq 0, \ \ x\in \Omega:=]0,x_{\mbox{max}}],
\end{eqnarray*}
 where $n(t,x)$ represents the solution to the truncated equation (\ref{trunc pbe fpt}). The existence and uniqueness of non-negative solutions for the truncated PBE (\ref{trunc pbe fpt}) has been shown in \cite{Ball:1990, Costa:1998, Dubovskii:1996, Stewart:1990}. In \cite{Dubovskii:1996, Escobedo:2003, GIRI:2010EXT, Stewart:1990}, it is proven that the sequence of solutions to the truncated problems converge weakly to the solution of the original problem in a weighted $L^1$ space as $x_{\mbox{max}} \to \infty$
 for certain classes of kernels. 

 Many numerical methods have been proposed to solve the truncated aggregation PBE (\ref{trunc pbe fpt}): finite element methods \cite{Everson:1997, Nicmanis:1998, Rigopoulos:2003}, finite volume methods \cite{Filbet:2007, Filbet:2004a, Filbet:2004}, stochastic methods \cite{Eibeck:2000, Eibeck:2001, Lee:2000}, moment methods \cite{Su:2007} and sectional methods \cite{J_Kumar:2007,J_Kumar:2006, Kumar:1996-1, Kumar:1996-2}.
By implementing most of these methods, we may have a quite
satisfactory results for the number density but not for moments.
However, the moment methods give opposite results. To have a
satisfactory information for the number density distribution as
well as some selected moments, the sectional methods have become
more useful nowadays. Several authors have proposed sectional
methods for aggregation PBE: S. Kumar and Ramkrishna
\cite{Kumar:1996-1, Kumar:1996-2}, J. Kumar et al.\
\cite{J_Kumar:2007, J_Kumar:2006} and Vanni \cite{Vanni:2002}. The
fixed pivot technique given by S. Kumar and Ramkrishna is the most
extensively used sectional method. This technique also efficiently
works for a multi-dimensional size variable
\cite{S_Kumarmulti:2007}.

 Recently J. Kumar and Warnecke \cite{JKUMARNM:2008} have published the numerical analysis of the fixed pivot technique for breakage PBEs. This case was simpler due to the linearity of that equation. However the convergence analysis of the technique was still open for aggregation PBEs (\ref{trunc pbe fpt}). This was a challenging task due to the non-linearity of the equation. So the purpose of this work is to demonstrate the missing convergence analysis of the fixed pivot technique for aggregation PBEs in the literature.

 Let us now briefly outline the contents of this paper. Along with the general idea of sectional methods, a concise review of the mathematical formulation of the fixed pivot technique is given in the following section. A theorem from Hundsdorfer and Verwer \cite{Hundsdorfer:2003} used in further analysis and the main result for the convergence of the fixed pivot technique are also stated in Section \ref{sectional method}. To show the convergence of the scheme for solving aggregation PBEs (\ref{trunc pbe fpt}), the consistency and Lipschitz conditions  are discussed in Sections \ref{consistency fpt} and \ref{convergence fpt}, respectively. Numerical simulations are performed in Section \ref{numerics fpt}. And, Section \ref{conclusion} states  some conclusions.

\section{The Sectional Methods}\label{sectional method}
 The mathematical formulation of sectional methods is reviewed from \cite[Section 2]{JKUMARNM:2008}. These methods calculate the total number of particles in finite number of cells. First of all, the continuous interval $\Omega
:= ]0, x_{\mbox{max}}]$ is divided into a finite number of cells defining size classes $\Lambda_i := ]x_{i-1/2}, x_{i+1/2}], i = 1, \ldots, I$.
Set
$$x_{1/2} = 0, \quad x_{I+1/2}= x_{\mbox{max}}, \quad  \Delta x_{\mbox{min}} \leq \Delta x_i = x_{i+1/2} - x_{i-1/2} \leq \Delta x.$$
For the purpose of later analysis we assume quasi uniformity of the grids, i.e.\
\begin{equation}
 \frac{\Delta x}{\Delta x_{\mbox{min}}}\leq C, \label{grid cond}
\end{equation}

where $C$ is a positive constant. The center point of each cell $x_i = (x_{i-1/2} + x_{i+1/2})/2$, $i = 1, \ldots, I$ is called pivot or
grid point. The integration of the truncated PBE (\ref{trunc pbe fpt}) over each cell gives a \emph{semi-discrete} system in $\mathbb{R}^I$
\begin{equation}\label{E:IntegatedBreakage}
\frac{d\mathbf{N}}{dt} = \mathbf{B} - \mathbf{D}, \hspace{.1in} \mbox{with} \hspace{.1in} \mathbf{N}(0) = \mathbf{N}^{\mbox{in}},
\end{equation}
where $\mathbf{N}^{\mbox{in}}, \mathbf{N}, \mathbf{B}, \mathbf{D} \in \mathbb{R}^I$. The $i$th components of vectors $\mathbf{N}, \mathbf{N}^{\mbox{in}}, \mathbf{B}$, and $\mathbf{D}$ are respectively, defined as
\begin{equation}\label{E:TotalNumberRepresentation}
N_i(t) = \int_{x_{i-1/2}}^{x_{i+1/2}} n(t, x) dx, \hspace{.1in} \mbox{with} \hspace{.1in} N_i^{\mbox{in}} = \int_{x_{i-1/2}}^{x_{i+1/2}} n^{\mbox{in}}(x) dx,
\end{equation}
\begin{equation} \label{E:B_i}
B_i=\frac{1}{2}\int_{x_{i-1/2}}^{x_{i+1/2}}\int_{0}^{x}K(x-y, y)n(t,x-y)n(t,y)dy dx.
\end{equation}
and
\begin{equation}\label{E:D_i}
D_i = \int_{x_{i-1/2}}^{x_{i+1/2}}\int_{0}^{x_{I+1/2}}K(x, y)n(t,y)n(t,x)dy dx.
\end{equation}
Here the vector $\mathbf{N}$ is formed by the vector of values of the step function obtained by $L^2$ projection of the exact solution $n$
into the space of step functions constant on each cell. Note that this projection error can easily be shown of second order, see section 5.2.2 in \cite{Giri:2010thesis}. Finally, by taking numerical approximations of $B_i$ and $D_i$ in terms of $N_i(t)$, the sectional methods give the
following discretized form
\begin{equation}\label{E:DiscretizedGenBreakage}
\frac{d\hat{\mathbf{N}}}{dt} = \hat{\mathbf{B}}(\hat{\mathbf{N}}) - \hat{\mathbf{D}}(\hat{\mathbf{N}})=:\hat{\mathbf {F}}(t,\hat{\mathbf{N}}), \hspace{.1in} \mbox{with} \hspace{.1in} \hat{\mathbf{N}}(0) = \mathbf{N}^{\mbox{in}},
\end{equation}

where $ \hat{\mathbf{N}}, \hat{\mathbf{B}}, \hat{\mathbf{D}}\in \mathbb{R}^I$. The total number of particles in $i$th
cell, $N_i(t)$ are numerically approximated by the $i$th component, $\hat{N}_i(t)$ of the vector $\hat{\mathbf{N}}$. The $i$th components of $\hat{\mathbf{B}}$ and $\hat{\mathbf{D}}$ are denoted by $\hat{B}_i$ and $\hat{D}_i$ respectively which are defined in spatially discretized system (\ref{resultant system}) obtained by the fixed pivot technique.

\subsection{The Fixed Pivot Technique} \label{FPT}
The fixed pivot technique relies on the following idea of birth modification. In \cite{Kumar:1996-1}, the truncated PBE (\ref{trunc pbe fpt}) is modified to
\begin{eqnarray}\label{fptmodi}
 \hspace{.6in}\frac{d}{dt} \int_{x_{i-1/2}}^{x_{i+1/2}}n(t,x)dx
 \approx &&\frac{1}{2}\int_{x_{i}}^{x_{i+1}}\lambda_i^+(x)\int_{0}^{x}K(x-y, y)n(t,x-y)n(t,y)dy dx\\
&&\nonumber+\frac{1}{2}\int_{x_{i-1}}^{x_{i}}\lambda_i^-(x)\int_{0}^{x}K(x-y, y)n(t,x-y)n(t,y)dy dx\\
&&\nonumber-\int_{x_{i-1/2}}^{x_{i+1/2}}\int_{0}^{x_{I+1/2}}K(x, y)n(t,y)n(t,x)dy dx.
\end{eqnarray}
where
\begin{equation}\label{values of lambda fpt}
 \lambda_i^{\pm}(x)=\frac{x-x_{i\pm 1}}{x_i-x_{i\pm 1}}, \ \ i = 1, \ldots, I.
\end{equation}
According to S. Kumar and Ramkrishna in \cite{Kumar:1996-1}, the first and second integral terms on the right hand side in (\ref{fptmodi}) set to be zero for $i=I$ and $i=1$, respectively.

Inserting the number density approximation $n(t,x)\approx \sum_{i=1}^{I}N_i(t)\delta(x-x_i)$, into the above equation, we get the following spatially discretized system
\begin{eqnarray}\label{resultant system}
 \hspace{.6in}\frac{dN_i(t)}{dt}= &&\sum_{x_i \leq x_j+x_k < x_{i+1}}^{j\geq k}\bigg(1-\frac{1}{2}\delta_{j,k}\bigg)\lambda_i^+(x_k+ x_j)K(x_k,x_{j})N_j(t)N_k(t)\\
&&+\sum_{x_{i-1} \leq x_j+x_k < x_{i}}^{j\geq k}\bigg(1-\frac{1}{2}\delta_{j,k}\bigg)\lambda_i^-(x_k+ x_j)K(x_k,x_{j})N_j(t)N_k(t)\nonumber\\
&&-N_i(t)\sum_{j=1}^{I}K(x_i, x_j)N_j(t),\nonumber\\
=&& \hat{B}_i-\hat{D}_i,\nonumber
\end{eqnarray}
where $\hat{B}_i$ and $\hat{D}_i$ represent the discretized birth and death terms, respectively in the $i\mbox{th}$ cell obtained from the fixed pivot technique. The
basic idea of the fixed pivot technique can be described
as follows. Assume that a new born particle of a size, which is not positioned at a
pivot point of any cell, appears due to the aggregation of two
smaller particles. The particle has to be assigned onto neighboring
pivot points in such a way that the particle number and mass are
conserved. This problem can be solved in a unique way. The resulting
technique gives very often quite satisfactory results. However, the
undesirable part is that the fixed pivot technique turns into
a zero order method on oscillatory and non-uniform random meshes for
aggregation problems.

It is important to mention here that at the last boundary cell S. Kumar and Ramkrishna simply set the first integral on the right hand side
in (\ref{fptmodi}) to be zero in their numerical computations. However, we have observed in our analysis for the aggregation problem that this setting at the end boundary cell reduces by one order the accuracy of the fixed pivot technique on uniform and non-uniform smooth meshes. To overcome this problem, we take an extra grid point $x_{I+1}$ at a $\Delta x_I$ distance away from the grid point $x_I$. In the computations contributions which are larger than $x_I$ are distributed to $x_I$ and $x_{I+1}$. This idea is used in Lemma \ref{consistency lem}. A similar modification should also be used at the first boundary cell for breakage problems in Lemma 3.2 of \cite{JKUMARNM:2008}. The present form of Lemma 3.2 in \cite{JKUMARNM:2008} is not correct at the first boundary cell and also reduces by one order accuracy of the fixed pivot technique for breakage PBEs on uniform and non-uniform smooth meshes.

It should be pointed out here that in this work we consider the following discrete $L_1$ norm 
\vspace{-0.70cm}
\begin{eqnarray*}
 \|\mathbf{N}\|=\sum_{i=1}^{I}|N_i|.
\end{eqnarray*}
We consider $\mathcal{C}^2([a,b])$ as a space of twice
continuously differentiable functions on $]a,b[$. Note that for
the sake of simplicity in our analysis we assume that the
aggregation kernel satisfies
\begin{equation}\label{cond on beta}
\hspace{.6in}K \in \mathcal{C}^2(]0,x_{\mbox{max}}]\times
]0,x_{\mbox{max}}]).
\end{equation}


As we move on to the subsequent sections, it will be helpful to revisit some
definitions and an existing theorem given in Hundsdorfer and Verwer \cite{Hundsdorfer:2003}.
These will be of use while discussing in detail the consistency and the convergence of the fixed pivot technique.

Let $\|\cdot\|$ denote any norm on $\mathbb{R}^I$.
\begin{definition}
The \textbf{spatial truncation error} is defined by the residual left by substituting the exact solution
$\mathbf{N}(t)$ into equation (\ref{E:DiscretizedGenBreakage}) as
\begin{equation}\label{E:SpatialTruncationError}
\mathbf{\sigma}(t) = \frac{d {\mathbf{N}}(t)}{d t} - \left(\hat{\mathbf{B}}\left({\mathbf{N}}(t)\right) -
\hat{\mathbf{D}}\left({\mathbf{N}}(t)\right)\right).
\end{equation}
The scheme (\ref{E:DiscretizedGenBreakage}) is called consistent of order $p$ if, for $\Delta x \to 0$,
\begin{eqnarray*}
\| \mathbf{\sigma}(t) \| = {\cal O}(\Delta x^p), \quad \mbox{uniformly for all } t, \, 0 \leq t \leq T.
\end{eqnarray*}
\end{definition}

\begin{definition}
The \textbf{global discretization error} is defined by
\begin{equation}\label{E:GlobalDiscretizationError}
 \mathbf{\epsilon}(t) = \mathbf{N}(t) -\hat{\mathbf{N}}(t).
\end{equation}
 The scheme (\ref{E:DiscretizedGenBreakage}) is called convergent of order $p$ if, for $\Delta x \to 0$,
\begin{eqnarray*}
\| \mathbf{\epsilon}(t) \| = {\cal O}(\Delta x^p), \quad \mbox{uniformly for all } t,\, 0 \leq t \leq T.
\end{eqnarray*}
\end{definition}

%
It is important that the solution obtained by the fixed pivot technique remains non-negative for all times. This can be easily shown by using the next well known theorem.
In the following theorem we write $\hat{\mathbf{M}}\geq 0$ for a vector $\hat{\mathbf{M}}\in \mathbb{R}^I$ if all of its components are non-negative.
\begin{theorem}\label{thm for positivity} (Hundsdorfer and Verwer \cite{Hundsdorfer:2003}).
Suppose that $\hat{\mathbf {F}}(t,\hat{\mathbf{M}})$ defined in (\ref{E:DiscretizedGenBreakage}) is continuous and satisfies the Lipschitz condition as
\begin{eqnarray*}
 \|\hat{\mathbf{F}}(t,\hat{\mathbf{P}})-\hat{\mathbf{F}}(t,{\hat{\mathbf{M}}})\|\leq L \|\hat{\mathbf{P}}-\hat{\mathbf{M}}\|\ \ \mbox{ for all}\ \ \hat{\mathbf{P}}, \hat{\mathbf{M}}\in \mathbb{R}^I.
\end{eqnarray*}
Then the solution of the semi-discrete system (\ref{E:DiscretizedGenBreakage}) is non-negative if and only if for any vector $\hat{\mathbf{M}}\in \mathbb{R}^I$ with $\hat{\mathbf{M}}\geq 0$
we have for any $i=1,\ldots,I$ and all $t\geq 0$ that $\hat{M}_i =0$ implies $\hat{F}_i(t,\hat{\mathbf{M}}) \geq 0.$
\end{theorem}
\begin{proof}
 The proof can be found in \cite{Hundsdorfer:2003}, Chap. $1$, Theorem $7.1$.
\end{proof}

Now we shall state the main result which helps us to show the convergence of the fixed pivot technique.
\begin{theorem}\label{convergence theorem}
Let us assume that the Lipschitz conditions on $\hat{\mathbf{B}}({\mathbf{N}}(t))$ and $\hat{\mathbf{D}}({\mathbf{N}}(t))$ are satisfied for $0 \leq t \leq T$  and for all $\mathbf{N}$, $\hat{\mathbf{N}}\in \mathbb{R}^I$ where $\mathbf{N}$ and $\hat{\mathbf{N}}$ are the projected exact and numerical solutions defined in (\ref{E:IntegatedBreakage}) and (\ref{E:DiscretizedGenBreakage}) respectively. Then a consistent discretization method is  also convergent and the convergence is of the same order as the consistency.
\end{theorem}
\begin{proof}
Using the equations (\ref{E:SpatialTruncationError}) and (\ref{E:GlobalDiscretizationError}) we have for $\mathbf{\epsilon}(t) = \mathbf{N}(t) -\hat{\mathbf{N}}(t)$
\begin{eqnarray*}
  \frac{d}{dt}\epsilon(t)=\sigma(t)+(\hat{\mathbf{B}}({\mathbf{N}})-\hat{\mathbf{B}}({\mathbf{\hat{N}}}))-(\hat{\mathbf{D}}({\mathbf{N}})-\hat{\mathbf{D}}({\mathbf{\hat{N}}})).
 \end{eqnarray*}
We then take the norm on both sides to get
\begin{eqnarray*}
  \frac{d}{dt}\|\epsilon(t)\|\leq\|\sigma(t)\|+\|(\hat{\mathbf{B}}({\mathbf{N}})-\hat{\mathbf{B}}({\mathbf{\hat{N}}}))\|+\|(\hat{\mathbf{D}}({\mathbf{N}})-\hat{\mathbf{D}}({\mathbf{\hat{N}}}))\|.
 \end{eqnarray*}

Integrating with respect to $t$ with $\epsilon(0)=0$ and using the Lipschitz conditions (\ref{lipschitz condition for B fpt})-(\ref{lipschitz condition for D fpt}) we obtain the estimates
\begin{eqnarray*}
 \|\epsilon(t)\|\leq \int_{0}^{t}\|\sigma(\tau)\|d \tau+ 2L \int_{0}^{t}\|\epsilon(\tau)\|d \tau.
\end{eqnarray*}

From this it follows by Gronwall's lemma that
\begin{equation}
 \|\epsilon(t)\|\leq \frac{e_h}{2L}[\exp{(2Lt)}-1],
\end{equation}
where
\begin{eqnarray*}
 e_h=\max_{0\leq t \leq T} \|\sigma(t)\|.
\end{eqnarray*}

If the scheme is consistent then $\lim_{h\to 0}e_h=0$. This completes the proof of the Theorem \ref{convergence theorem}.
\end{proof}

Note that the proof of the Theorem \ref{convergence theorem} is motivated by a convergence result in Linz \cite{Linz:1975}.

To fulfill the requirements of Theorem \ref{convergence theorem}, for the convergence of the fixed pivot technique we need to prove that the scheme is consistent as well as the birth $\hat{\mathbf{B}}({\mathbf{N}}(t))$ and death $\hat{\mathbf{D}}({\mathbf{N}}(t))$ terms satisfy the Lipschitz conditions.

\section{Consistency}\label{consistency fpt}
We need the following lemma to investigate the consistency of the fixed pivot technique for aggregation PBEs (\ref{trunc pbe fpt}) which is the same as Lemma 3.2 in \cite{JKUMARNM:2008} except at the boundary cells.

\begin{lemma}\label{consistency lem}
Consider a function $f\in \mathcal{C}^2([0,x_{\mbox{max}}])$ and a cell centered partitioning of the domain $[0,x_{\mbox{max}}]$ given as $0=x_{1-1/2}<\ldots<x_{i-1/2}<x_{i+1/2}<\ldots<x_{I+1/2}=x_{\mbox{max}}$ with pivot points $x_i=(x_{i-1/2}+x_{i+1/2})/2$ and a bound $\Delta x \geq \Delta x_i=(x_{i+1/2}-x_{i-1/2})$ for all $i$. If $\lambda_i^+(x)$ and $\lambda_i^-(x)$ are given by the definition (\ref{values of lambda fpt}), then the following expressions can be obtained for the modification error
\begin{eqnarray*}
\mathfrak{I}_i(f) =&&\int_{x_{i-1/2}}^{x_{i+1/2}} f(x)dx-\int_{x_{i}}^{x_{i+1}} \lambda_i^+(x) f(x)dx -\int_{x_{i-1}}^{x_{i}} \lambda_i^-(x) f(x)dx\\
=&&\frac{f(x_i)}{2}\bigg[\Delta x_i-\bigg(\frac{\Delta x_{i-1}+\Delta x_{i+1}}{2}\bigg)\bigg]\\
&&-\frac{f'(x_i)}{12}\bigg[(\Delta x_{i+1}-\Delta x_{i-1})\bigg\{\Delta x_i+\bigg(\frac{\Delta x_{i-1}+\Delta x_{i+1}}{2}\bigg)\bigg\}\bigg]\\
&&+{\cal O}(\Delta x^3), \ \ \mbox{for}\ \ i=2,\ldots, I-1,
\end{eqnarray*}

\begin{eqnarray*}
\mathfrak{I}_i(f) &&=\int_{x_{i-1/2}}^{x_{i+1/2}} f(x)dx-\int_{x_{i}}^{x_{i+1}} \lambda_i^+(x) f(x)dx - \int_{x_{i-1}}^{x_{i}} \lambda_i^-(x) f(x)dx\\
&& = \frac{f(x_i)}{4}[\Delta x_i-\Delta x_{i-1}]+{\cal O}(\Delta x^2),\ \ \mbox{for}\ \ i=I,
\end{eqnarray*}
\begin{eqnarray*}
\mathfrak{I}_i(f) &&=\int_{x_{i-1/2}}^{x_{i+1/2}} f(x)dx-\int_{x_{i}}^{x_{i+1}} \lambda_i^+(x) f(x)dx\\
 &&= \frac{f(x_i)}{4}[3 \Delta x_i-\Delta x_{i+1}]+{\cal O}(\Delta x^2),\ \ \mbox{for}\ \ i=1.
\end{eqnarray*}
\end{lemma}
Note that $x_{I+1}$ is the extra grid point introduced in section \ref{FPT}.
%
%
%
\begin{proof}
First, we consider the modification error $\mathfrak{I}_i$ for $i = 2,\ldots I-1$ as follows
$$
\mathfrak{I}_i(f) = \int_{x_{i-1/2}}^{x_{i+1/2}} f(x)dx - \bigg(\int_{x_i}^{x_{i+1}} \lambda_i^+(x) f(x) dx +
\int_{x_{i-1}}^{x_i} \lambda_i^-(x) f(x) dx \bigg).
$$
 Taylor series expansion of $f(x)$ about $x_i$ in $\mathfrak{I}_i$ yields
\begin{eqnarray*}
\mathfrak{I}_i(f) = &&  f(x_i) \bigg[\Delta x_i - \bigg(\int_{x_i}^{x_{i+1}} \lambda_i^+(x) dx + \int_{x_{i-1}}^{x_i}
\lambda_i^-(x) dx\bigg) \bigg] \\ && - f'(x_i) \bigg(\int_{x_i}^{x_{i+1}} \lambda_i^+(x)(x-x_i) dx +
\int_{x_{i-1}}^{x_i} \lambda_i^-(x)(x-x_i)dx\bigg) + {\cal O}(\Delta x^3).
\end{eqnarray*}
Substituting the values of $\lambda_i^+$ and $\lambda_i^-$  from (\ref{values of lambda fpt}) into the preceding equation, we obtain
\begin{eqnarray} \label{E:LS}
\mathfrak{I}_i(f) = && f(x_i) [\Delta x_i - \frac{1}{2} (x_{i+1} - x_{i-1} )]  \\
&&\nonumber -\frac{f'(x_i)}{6} \bigg[(x_{i+1} - x_{i-1}) \{ (x_{i+1} - x_i) - (x_i-x_{i-1})\}\bigg] + {\cal O}(\Delta x^3).
\end{eqnarray}
For the cell centered grids, i.e.\ $x_i = (x_{i-1/2}+x_{i+1/2})/2$, the equation (\ref{E:LS}) becomes
\begin{eqnarray*}
\mathfrak{I}_i(f) = && \frac{f(x_i)}{2} [\Delta x_i - \frac{1}{2} (\Delta x_{i+1} + \Delta x_{i-1}
)]\\ && - \frac{f'(x_i)}{12} \bigg[( \Delta x_{i+1} - \Delta x_{i-1})\bigg\{ \Delta
x_i + \bigg(\frac{\Delta x_{i-1} + \Delta x_{i+1}} {2} \bigg) \bigg\}\bigg] + {\cal O}(\Delta x^3).
\end{eqnarray*}
The expression of modification error at last boundary cell i.e.\
$i = I$ is given by
\begin{eqnarray*}
\mathfrak{I}_I(f) =  \int_{x_{I-1/2}}^{x_{I+1/2}} f(x) dx- \int_{x_{I}}^{x_{I+1}} \lambda_I^+(x) f(x) dx - \int_{x_{I-1}}^{x_{I}} \lambda_I^-(x) f(x) dx.
\end{eqnarray*}

Applying the midpoint, left rectangle and right rectangle rules in first, second and thirds integrals
respectively, we get
\begin{eqnarray*}
\mathfrak{I}_I(f) = && f(x_I) \Delta x_I -\bigg[\lambda_I^+(x_I) f(x_I) (x_{I+1}-x_{I})-\frac{f(x_I)}{2} (x_{I+1}-x_{I})\bigg]\\
&&- \bigg[\lambda_I^-(x_I) f(x_I) (x_I-x_{I-1})-\frac{f(x_I)}{2} (x_{I}-x_{I-1})\bigg] + {\cal O}(\Delta x^2)\\
=&& f(x_I) \Delta x_I - \frac{f(x_I)}{2} (x_{I+1}-x_{I}) - \frac{f(x_I)}{2} (x_{I}-x_{I-1})+ {\cal O}(\Delta x^2)\\
               = && f(x_I) \left[ \Delta x_I -\frac{1}{2}\Delta x_I- \frac{1}{4}(\Delta x_I + \Delta x_{I-1}) \right] + {\cal O}(\Delta
               x^2)\\
               = && \frac{f(x_I)}{4}\left(\Delta x_I - \Delta x_{I-1}\right) + {\cal O}(\Delta
               x^2).
\end{eqnarray*}
Now we consider $i=1$
\begin{eqnarray*}
\mathfrak{I}_1(f)=\int_{x_{1/2}}^{x_{3/2}} f(x) dx- \int_{x_{1}}^{x_{2}} \lambda_I^+(x) f(x) dx.
\end{eqnarray*}
We apply midpoint and left rectangle rules in first and second integrals
respectively to obtain
\begin{eqnarray*}
\mathfrak{I}_1(f) = && f(x_1)\Delta x_1-\bigg[\lambda_1^+(x_1) f(x_1) (x_{2}-x_{1})-\frac{f(x_1)}{2}(x_{2}-x_{1})\bigg] + {\cal O}(\Delta x^2)\\
=&& f(x_1)\bigg[\Delta x_1- \frac{1}{4}(\Delta x_1+ \Delta x_2)\bigg]\\
= && \frac{f(x_1)}{4}(3 \Delta x_1-\Delta x_2).
\end{eqnarray*}
Hence the proof of Lemma \ref{consistency lem} is completed.
\end{proof}

Now the four main sections are described to inquire the consistency of the fixed pivot technique.
   First, we evaluate the discretization error in the birth and death terms in Sections \ref{order of the birth term} and \ref{Order of the death term}, respectively. Then one can summarize all terms in section \ref{summary of all terms}. Finally, the five different types of meshes are considered to evaluate the local discretization error in Section \ref{Meshes}.

\subsection{Discretization error in the birth term}\label{order of the birth term}
 The integrated birth term of aggregation PBE (\ref{trunc pbe fpt}) over $i$th cell can be written as follows
\begin{eqnarray*}
B_i=\frac{1}{2}\int_{x_{i-1/2}}^{x_{i+1/2}}\int_{0}^{x}K(x-y, y)n(t,x-y)n(t,y)dy dx.
\end{eqnarray*}
Let us denote
\begin{equation}\label{def of f}
f(t,x)=\frac{1}{2}\int_{0}^{x}K(x-y, y)n(t,x-y)n(t,y)dy.
\end{equation}
\subsubsection{Birth term on internal cells}
Considering $i=2,\ldots, I-1$ and using Lemma \ref{consistency lem}, we can rewrite $B_i$ as
\begin{eqnarray*}
B_i=&&\frac{1}{2}\int_{x_{i}}^{x_{i+1}}\lambda_i^+(x)\int_{0}^{x}K(x-y,y)n(t,x-y)n(t,y)dy dx\\
&&+\frac{1}{2}\int_{x_{i-1}}^{x_{i}}\lambda_i^-(x)\int_{0}^{x}K(x-y, y)n(t,x-y)n(t,y)dy dx + \mathfrak{J}_i(f).
\end{eqnarray*}
By changing the order of integration of the first two terms on the right hand side, we obtain
\vspace{-.40cm}
\begin{eqnarray}\label{birth term}
\hspace{.6in}B_i=&&\frac{1}{2}\sum_{j=1}^{i-1}\int_{x_{j-1/2}}^{x_{j+1/2}}\int_{x_i}^{x_{i+1}}\lambda_i^+(x)K(x-y,y)n(t,x-y)n(t,y) dx dy\\
&&+\frac{1}{2}\int_{x_{i-1/2}}^{x_{i}}\int_{x_i}^{x_{i+1}}\lambda_i^+(x)K(x-y,y)n(t,x-y)n(t,y) dx dy\nonumber\\
&&+\frac{1}{2}\int_{x_i}^{x_{i+1}}\int_{y}^{x_{i+1}}\lambda_i^+(x)K(x-y, y)n(t,x-y)n(t,y)dx dy\nonumber\\
&&+\frac{1}{2}\sum_{j=1}^{i-2}\int_{x_{j-1/2}}^{x_{j+1/2}}\int_{x_{i-1}}^{x_{i}}\lambda_i^-(x)K(x-y,y)n(t,x-y)n(t,y) dx dy\nonumber\\
&&+\frac{1}{2}\int_{x_{i-3/2}}^{x_{i-1}}\int_{x_{i-1}}^{x_{i}}\lambda_i^-(x)K(x-y, y)n(t,x-y)n(t,y) dx dy\nonumber\\
&&\nonumber+\frac{1}{2}\int_{x_{i-1}}^{x_{i}}\int_{y}^{x_{i}}\lambda_i^-(x)K(x-y, y)n(t,x-y)n(t,y) dx dy + \mathfrak{J}_i(f).
\end{eqnarray}
Let us denote the integral terms on the right hand side in (\ref{birth term}) by $I_1, \ldots, I_6$ respectively and calculate them separately.
\paragraph{The integrals $I_6$ and $I_3$}
First we consider the integral $I_6$ as follows
\begin{eqnarray*}
I_6=\frac{1}{2}\int_{x_{i-1}}^{x_{i}}g_i(t,y)dy,
\end{eqnarray*}
where
\vspace{-.40cm}
\begin{equation} \label{g}
g_i(t,y):= \int_{y}^{x_{i}}\lambda_i^-(x)K(x-y, y)n(t,x-y)n(t,y) dx.
\end{equation}

Using Taylor series expansion of $g(t, y)$ with respect to $y$ about $x_{i-1}$, we get
\begin{equation}\label{first I_6}
\hspace{.6in}I_6=g_i(t,x_{i-1})\frac{(\Delta x_{i-1}+ \Delta x_{i})}{4}+ g'_{i}(t,x_{i-1})\frac{(\Delta x_{i-1}+ \Delta x_{i})^2}{16}+ \ldots,
\end{equation}
where
\vspace{-.70cm}
\begin{eqnarray*}
 g'_i(t,x_{i-1})=\frac{\partial g_i}{\partial y}(t,x_{i-1})=\frac{\partial g_i}{\partial y}(t, y)|_{y=x_{i-1}}.
\end{eqnarray*}

Now we calculate $g'_{i}(t,x_{i-1})$ using the Leibniz rule to differentiate under an integral. Then, we apply the left rectangle rule to the integrals in $g'_{i}(t,x_{i-1})$ and $g_i(t,x_{i-1})$ defined in (\ref{g}). Further, we insert $\lambda_i^-(x_{i-1})=0$ and $K(0, x_{i-1})=0$ to obtain
\begin{eqnarray*}
g_i(t,x_{i-1})= 0+{\cal O}(\Delta x^2) \ \ \mbox{and}\ \ g'_{i}(t,x_{i-1})=0+{\cal O}(\Delta x^2).
\end{eqnarray*}

%
Thus, by substituting the value of $g'_{i}(t,x_{i-1})$ and $g_i(t,x_{i-1})$ in (\ref{first I_6}), we have
\begin{equation}\label{second I_6}
I_6= 0 + {\cal O}(\Delta x^3).
\end{equation}

Similarly, we can evaluate the integral $I_3$ as
\begin{equation}\label{second I_3}
I_3= 0 + {\cal O}(\Delta x^3).
\end{equation}

\paragraph{The integrals $I_1$ and $I_4$}

Let us next consider the first integral $I_1$ from equation (\ref{birth term}) and apply the midpoint rule to the outer integral. Furthermore we can use the relationship $N_j(t)=n(t,x_j)\Delta x_j + {\cal O}(\Delta x^3)$ for the midpoint rule and substitute $x-x_j=x'$ to obtain
\begin{equation} \label{first I_1}
I_1=\frac{1}{2}\sum_{j=1}^{i-1}N_j(t)\int_{x_i-x_j}^{x_{i+1}-x_j}\lambda_i^+(x'+ x_j)K(x',x_{j})n(t,x') dx' +{\cal O}(\Delta x^3).
\end{equation}

We define $l_{i,j}$ and $\gamma_{i,j}$ to be those indices such that the following hold
\begin{equation}\label{defi of index fpt}
 x_i-x_j \in \Lambda_{l_{i,j}}\ \ \mbox{and} \ \ \gamma_{i,j}:=H [(x_i-x_j)- x_{l_{i,j}}]
\end{equation}
\vspace{-.70cm}
where
\begin{eqnarray*}
H(x):=
\left\{
    \begin{array}{ll}
        1  & \mbox{if } x>0, \\
        -1 & \mbox{if } x\leq 0.
    \end{array}
\right.
\end{eqnarray*}

We will use the convention of Riemann integration that $\int_{a}^{b}f(x)dx=-\int_{b}^{a}f(x)dx$ and the
equation (\ref{first I_1}) can be rewritten as
\begin{eqnarray*}
I_1=&&\frac{1}{2}\sum_{j=1}^{i-1}N_j(t)\int_{x_i-x_j}^{x_{l_{i,j}+ \frac{1}{2}\gamma_{i,j}}}\lambda_i^+(x'+ x_j)K(x',x_{j})n(t,x') dx'\\
&&+\frac{1}{2}\sum_{j=1}^{i-1}N_j(t)\sum_{k=l_{i,j}+ \frac{1}{2}(1+\gamma_{i,j})}^{l_{i+1,j}+ \frac{1}{2}(\gamma_{i+1,j}-1)}\int_{x_{k-{1/2}}}^{x_{k+{1/2}}}\lambda_i^+(x'+ x_j)K(x',x_{j})n(t,x') dx'\\
&&+\frac{1}{2}\sum_{j=1}^{i-1}N_j(t)\int_{x_{l_{i+1,j}+ \frac{1}{2}\gamma_{i+1,j}}}^{x_{i+1}-x_j}\lambda_i^+(x'+ x_j)K(x',x_{j})n(t,x') dx'+{\cal O}(\Delta x^3).
\end{eqnarray*}

Let $p$ be the total number of terms in the following sum
\begin{eqnarray*}
\sum_{k={l_{i,j}+\frac{1}{2}(\gamma_{i,j}+1)}}^{l_{i+
1,j}+\frac{1}{2}(\gamma_{i+1,j}-1)}\int_{x_{k-1/2}}^{x_{k+1/2}}\lambda_i^+(x'+
x_j)K(x',x_j)n(t,x')dx'.
\end{eqnarray*}
In particular, let $p:= \#  \{ n:
l_{i,j}+\frac{1}{2}(\gamma_{i,j}+1) \leq n \leq
l_{i+1,j}+\frac{1}{2}(\gamma_{i+1,j}-1)\}$
 and set
 \vspace{-.20cm}
$$k_1:=l_{i,j}+\frac{1}{2}(\gamma_{i,j}+1),\ \ k_2:=k_1+1,\ldots,\ \ k_{p-1}:=k_1+(p-2).$$

Next, we shall show that $p$ is finite and can be estimated by a constant which is independent of the grid size.
By using the definition of the indices $l_{i,j}$ and $\gamma_{i,j}$ in (\ref{defi of index fpt}), we can estimate
\vspace{-.40cm}
$$(p-2)\Delta x_{\mbox{min}}\leq \Delta x_{k_2}+ \Delta x_{k_3}+ \ldots + \Delta x_{k_{p-1}} \leq \frac{1}{2}(\Delta x_{i}+\Delta x_{i+1})\leq \Delta x$$

which implies using the assumption of quasi uniformity (\ref{grid cond}) that
$$(p-2)\leq \frac{\Delta x}{\Delta x_{\mbox{min}}}\leq C \Rightarrow  p \leq C+2.$$

This means the above sum has uniformly bounded finite number of
terms. So we can apply the midpoint rule to the integral in second term on
the right hand side and use $N_k(t)=n(t,x_k)\Delta x_k + {\cal O}(\Delta x^3)$ to get
\begin{eqnarray} \label{second I_1}
\hspace{.6in}I_1=&&\frac{1}{2}\sum_{j=1}^{i-1}N_j(t)\int_{x_i-x_j}^{x_{l_{i,j}+ \frac{1}{2}\gamma_{i,j}}}\lambda_i^+(x'+ x_j)K(x',x_{j})n(t,x') dx'\\
&&+\frac{1}{2}\sum_{j=1}^{i-1}N_j(t)\sum_{x_i \leq x_j+x_k < x_{i+1}}\lambda_i^+(x_k+ x_j)K(x_k,x_{j})N_k(t)\nonumber\\
&&+\frac{1}{2}\sum_{j=1}^{i-1}N_j(t)\int_{x_{l_{i+1,j}+ \frac{1}{2}\gamma_{i+1,j}}}^{x_{i+1}-x_j}\lambda_i^+(x'+ x_j)K(x',x_{j})n(t,x') dx'\nonumber\\
&&-\frac{1}{12}\sum_{j=1}^{i-1}N_j(t)\sum_{k=l_{i,j}+ \frac{1}{2}(1+\gamma_{i,j})}^{l_{i+1,j}+ \frac{1}{2}(\gamma_{i+1,j}-1)}\frac{{\Delta x_k}^3}{\Delta x_{i}+\Delta x_{i+1}}\frac{\partial}{\partial x'}\{K(x_k,x_{j})n(t,x_k)\}\nonumber\\
&&\nonumber+{\cal O}(\Delta x^3).
\end{eqnarray}

Further we proceed as in $I_1$ to evaluate the fourth integral $I_4$ on the right hand side from equation (\ref{birth term}) as
\begin{eqnarray}\label{second I_4}
\hspace{.6in}I_4=&&\frac{1}{2}\sum_{j=1}^{i-2}N_j(t)\int_{x_{i-1}-x_j}^{x_{l_{i-1,j}+ \frac{1}{2}\gamma_{i-1,j}}}\lambda_i^-(x'+ x_j)K(x',x_{j})n(t,x') dx'\\
&&+\frac{1}{2}\sum_{j=1}^{i-2}N_j(t)\sum_{x_{i-1} \leq x_j+x_k < x_{i}}\lambda_i^-(x_k+ x_j)K(x_k,x_{j})N_k(t)\nonumber\\
&&+\frac{1}{2}\sum_{j=1}^{i-2}N_j(t)\int_{x_{l_{i,j}+ \frac{1}{2}\gamma_{i,j}}}^{x_{i}-x_j}\lambda_i^-(x'+ x_j)K(x',x_{j})n(t,x') dx'\nonumber\\
&&+\frac{1}{12}\sum_{j=1}^{i-2}N_j(t)\sum_{k=l_{i-1,j}+ \frac{1}{2}(1+\gamma_{i-1,j})}^{l_{i,j}+ \frac{1}{2}(\gamma_{i,j}-1)}\frac{{\Delta x_k}^3}{\Delta x_{i}+\Delta x_{i-1}}\frac{\partial}{\partial x'}\{K(x_k,x_{j})n(t,x_k)\}\nonumber\\
&&\nonumber+{\cal O}(\Delta x^3).
\end{eqnarray}

\paragraph{The integrals $I_2$ and $I_5$}

To compute the integral $I_2$ in equation (\ref{birth term}), we apply the right rectangle rule to the outer integral and
use $N_i(t)=n(t,x_i)\Delta x_i + {\cal O}(\Delta x^3)$ then split it into two integrals as follows
\begin{eqnarray*}
 I_2=&&\frac{1}{2}\int_{x_i}^{x_{i+1}}\lambda_i^+(x)K(x-x_i,x_i)n(t,x-x_i)N_i(t) dx\\
 &&-\frac{1}{4}\int_{x_i}^{x_{i+1}}\lambda_i^+(x)K(x-x_i,x_i)n(t,x-x_i)N_i(t) dx+ {\cal O}(\Delta x^3).
\end{eqnarray*}

 We use the left rectangle rule in the second integral on the right hand side and then put $K(0,x_i)=0$ which gives
us a third order term. Therefore, we obtain by substituting $x-x_i=x'$ in the first integral
\vspace{-.20cm}
\begin{eqnarray*}
I_2=\frac{1}{2}N_i(t)\int_{0}^{x_{i+1}-x_i}\lambda_i^+(x'+ x_i)K(x',x_{i})n(t,x') dx' +{\cal O}(\Delta x^3),
\end{eqnarray*}
By using the indices $l_{i,j}$ and $\gamma_{i,j}$ defined in equation (\ref{defi of index fpt}), again we proceed in the same way as in $I_1$ to obtain
\begin{eqnarray}\label{second I_2}
\hspace{.6in}I_2=&&\frac{1}{2}N_i(t)\sum_{x_i+x_k<x_{i+1}}\lambda_i^+(x_k+ x_i)K(x_k,x_{i})N_k(t)\\
&&\nonumber+\frac{1}{2}N_i(t)\int_{x_{l_{i+1,i}+\frac{1}{2}\gamma_{i+1,i}}}^{x_{i+1}-x_i}\lambda_i^+(x'+ x_i)K(x',x_{i})n(t,x') dx' +{\cal O}(\Delta x^3).
\end{eqnarray}
Before combining all these integrals together, we also need to
discretize the integral $I_5$ defined in equation (\ref{birth term}). Finally, analogous to the integral $I_2$, we can evaluate $I_5$ as
\vspace{-.20cm}
\begin{eqnarray}\label{second I_5}
\hspace{.6in} I_5=&&\frac{1}{2}N_{i-1}(t)\sum_{x_{i-1}+x_k<x_i}\lambda_i^-(x_k+ x_{i-1})K(x_k,x_{i-1})N_k(t)\\
&&+\frac{1}{2}N_{i-1}(t)\int_{x_{l_{i,i-1}+\frac{1}{2}\gamma_{i,i-1}}}^{x_{i}-x_{i-1}}\lambda_i^-(x'+ x_{i-1})K(x',x_{i-1})n(t,x') dx'\nonumber\\
&&\nonumber +{\cal O}(\Delta x^3).
\end{eqnarray}
 Collecting together (\ref{second I_6}), (\ref{second I_3}), (\ref{second I_1}), (\ref{second I_4}), (\ref{second I_2}), (\ref{second I_5}) and substituting
all these terms into equation (\ref{birth term}), we obtain
\vspace{-.20cm}
\begin{eqnarray}\label{birth term 2}
B_i=\hspace{3.4in}&\\
 \left.\begin{array}{lcr}
\sum_{x_i \leq x_j+x_k < x_{i+1}}^{j\geq k}\bigg(1-\frac{1}{2}\delta_{j,k}\bigg)\lambda_i^+(x_k+ x_j)K(x_k,x_{j})N_j(t)N_k(t)\phantom{\Bigg|}\\
+\sum_{x_{i-1} \leq x_j+x_k < x_{i}}^{j\geq
k}\bigg(1-\frac{1}{2}\delta_{j,k}\bigg)\lambda_i^-(x_k+
x_j)K(x_k,x_{j})N_j(t)N_k(t)\phantom{\Bigg|}
\end{array}\right\}& = \hat{B}_i
\nonumber\\
\left.\begin{array}{l}
+\frac{1}{2}\sum_{j=1}^{i-1}N_j(t)\int_{x_i-x_j}^{x_{l_{i,j}+ \frac{1}{2}\gamma_{i,j}}}\lambda_i^+(x'+ x_j)K(x',x_{j})n(t,x') dx'\phantom{\Bigg|}\\
+\frac{1}{2}\sum_{j=1}^{i}N_j(t)\int_{x_{l_{i+1,j}+ \frac{1}{2}\gamma_{i+1,j}}}^{x_{i+1}-x_j}\lambda_i^+(x'+ x_j)K(x',x_{j})n(t,x') dx'\phantom{\Bigg|}\\
+\frac{1}{2}\sum_{j=1}^{i-2}N_j(t)\int_{x_{i-1}-x_j}^{x_{l_{i-1,j}+ \frac{1}{2}\gamma_{i-1,j}}}\lambda_i^-(x'+ x_j)K(x',x_{j})n(t,x') dx'\phantom{\Bigg|}\\
+\frac{1}{2}\sum_{j=1}^{i-1}N_j(t)\int_{x_{l_{i,j}+
\frac{1}{2}\gamma_{i,j}}}^{x_{i}-x_j}\lambda_i^-(x'+
x_j)K(x',x_{j})n(t,x') dx'\phantom{\Bigg|}
\end{array}\right\}& =: E
\nonumber\\
\left.\begin{array}{rcl}
-\frac{1}{12}\sum_{j=1}^{i-1}N_j(t)\sum_{k=l_{i,j}+
\frac{1}{2}(1+\gamma_{i,j})}^{l_{i+1,j}+
\frac{1}{2}(\gamma_{i+1,j}-1)}
\frac{{\Delta x_k}^3}{\Delta x_{i}+\Delta x_{i+1}}\frac{\partial}{\partial x'}\{K(x_k,x_{j})n(t,x_k)\}\phantom{\Bigg|}\\
+\frac{1}{12}\sum_{j=1}^{i-2}N_j(t)\sum_{k=l_{i-1,j}+
\frac{1}{2}(1+\gamma_{i-1,j})}^{l_{i,j}+
\frac{1}{2}(\gamma_{i,j}-1)} \frac{{\Delta x_k}^3}{\Delta
x_{i}+\Delta x_{i-1}}\frac{\partial}{\partial
x'}\{K(x_k,x_{j})n(t,x_k)\}\phantom{\Bigg|}
\end{array}\right\} & =: E'
\nonumber\\
\nonumber+{\cal O}(\Delta x^3)+ \mathfrak{J}_i(f).\hspace{3.4in}
\end{eqnarray}

The first two terms on the right hand side are exactly the fixed pivot discretization $\hat{B}_i$ from
(\ref{resultant system}). Let us denote the sum from third to sixth terms and the difference
of remaining two terms on the right hand side by $E$ and $E'$,
respectively. Doing so and substituting $\mathfrak{J}_i(f)$ from Lemma \ref{consistency lem}, equation (\ref{birth term 2}) can be written as for $i=2, \ldots, I-1 $
\begin{eqnarray}\label{final birth term for i}
\hspace{.6in}B_i=&&\hat{B}_i+ E+ E'+\frac{f(x_i)}{2}\bigg[\Delta x_i-\bigg(\frac{\Delta x_{i-1}+\Delta x_{i+1}}{2}\bigg)\bigg]\\
&&\nonumber-\frac{f'(x_i)}{12}\bigg[(\Delta x_{i+1}-\Delta x_{i-1})\bigg\{\Delta x_i+\bigg(\frac{\Delta x_{i-1}+\Delta x_{i+1}}{2}\bigg)\bigg\}\bigg]+{\cal O}(\Delta x^3).
\end{eqnarray}

\subsubsection{Birth term on boundary cells}
Now we evaluate the order of the error in birth term on boundary
cells. We will treat the boundary cells separately. Therefore,
first we consider the birth term for $i=1$.
Using the Lemma \ref{consistency lem} and changing the order of integration we get
\begin{eqnarray*}
B_1=&&\frac{1}{2}\int_{x_{1/2}}^{x_{1}}\int_{x_1}^{x_2}\lambda_1^+(x)K(x-y, y)n(t,x-y)n(t,y) dx dy\\
&&+\frac{1}{2}\int_{x_1}^{x_{1+1/2}}\int_{y}^{x_2}\lambda_1^+(x)K(x-y, y)n(t,x-y)n(t,y) dx dy\\
&&+\frac{1}{2}\int_{x_{1+1/2}}^{x_{2}}\int_{y}^{x_2}\lambda_1^+(x)K(x-y, y)n(t,x-y)n(t,y) dx dy + \mathfrak{J}_1(f).
\end{eqnarray*}
We apply the right rectangle rule to the outer and inner integrals
respectively of the first term as well as to the outer integral of
the third term. Both of these terms are of second order
separately. Also we use the left rectangle rule to the outer
integral of the second term to get
\begin{eqnarray*}
B_1=\frac{1}{4}\Delta x_1\int_{x_1}^{x_2}\lambda_1^+(x)K(x-x_1, x_1)n(t,x-x_1)n(t,x_1) dx+ {\cal O}(\Delta x^2)  + \mathfrak{J}_1(f).
\end{eqnarray*}

Again by applying the right rectangle rule and putting $N_1(t)=n(t,x_1)\Delta x_1 + {\cal O}(\Delta x^3)$, we obtain
\begin{eqnarray*}
B_1= \frac{1}{2}N_1(t)\sum_{x_1+x_k<x_{2}}\lambda_i^+(x_k+ x_1)K(x_k,x_{1})N_k(t) + {\cal O}(\Delta x^2)  + \mathfrak{J}_1(f).
\end{eqnarray*}

In terms of fixed pivot discretization we have
\begin{eqnarray}\label{birth term for 11}
B_1= \hat{B}_1 + \frac{f(x_1)}{4}[3\Delta x_1-\Delta x_{2}]+
{\cal O}(\Delta x^2),
\end{eqnarray}

where $f$ is defined in Eq. (\ref{def of f}).
Then the application of the right rectangle rule to the integral in (\ref{def of f}) gives $f(x_1)={\cal O}(\Delta x)$. Substituting this value of $f(x_1)$ in (\ref{birth term for 11}), we obtain
\begin{eqnarray}\label{birth term for 1}
B_1= \hat{B}_1 + {\cal O}(\Delta x^2).
\end{eqnarray}

Finally, in case of last boundary cell i.e.\ $i=I$, we use the Lemma \ref{consistency lem} and change the order of integration as in $B_1$. Further we split each integral into two parts to estimate
\begin{eqnarray*}
B_I=&&\frac{1}{2}\sum_{k=1}^{I-1}\int_{x_{k-1/2}}^{x_{k+1/2}}\int_{x_{I}}^{x_{I+1}}\lambda_I^+(x)K(x-y, y)n(t,x-y)n(t,y) dx dy\\
&&+\frac{1}{2}\int_{x_{I-1/2}}^{x_{I}}\int_{x_{I}}^{x_{I+1}}\lambda_I^+(x)K(x-y, y)n(t,x-y)n(t,y) dx dy\\
&&+\frac{1}{2}\int_{x_{I}}^{x_{I+1/2}}\int_{y}^{x_{I+1}}\lambda_I^+(x)K(x-y, y)n(t,x-y)n(t,y) dx dy\\
&&+\frac{1}{2}\int_{x_{I+1/2}}^{x_{I+1}}\int_{y}^{x_{I+1}}\lambda_I^+(x)K(x-y, y)n(t,x-y)n(t,y) dx dy\\
&&+\frac{1}{2}\sum_{k=1}^{I-2}\int_{x_{k-1/2}}^{x_{k+1/2}}\int_{x_{I-1}}^{x_I}\lambda_I^-(x)K(x-y, y)n(t,x-y)n(t,y) dx dy\\
&&+\frac{1}{2}\int_{x_{I-3/2}}^{x_{I-1}}\int_{x_{I-1}}^{x_I}\lambda_I^-(x)K(x-y, y)n(t,x-y)n(t,y) dx dy\\
&&+\frac{1}{2}\int_{x_{I-1}}^{x_{I-1/2}}\int_{y}^{x_I}\lambda_I^-(x)K(x-y, y)n(t,x-y)n(t,y) dx dy\\
&&+\frac{1}{2}\int_{x_{I-1/2}}^{x_{I}}\int_{y}^{x_I}\lambda_I^-(x)K(x-y, y)n(t,x-y)n(t,y) dx dy + \mathfrak{J}_I(f).
\end{eqnarray*}

By applying the midpoint, right, left and right rectangle rules to the outer integral of the first and fifth, second and sixth, third and seventh as well as fourth and eighth terms respectively. Then after substituting $N_k(t)=n(t,x_k)\Delta x_k + {\cal O}(\Delta x^3)$, we obtain
\begin{eqnarray*}
B_I=&&\frac{1}{2}\sum_{k=1}^{I-1}N_k(t)\int_{x_{I}}^{x_{I+1}}\lambda_I^+(x)K(x-x_k, x_k)n(t,x-x_k) dx  \\
&& + \frac{1}{2} N_{I}(t)\int_{x_{I}}^{x_{I+1}}\lambda_I^+(x)K(x-x_{I}, x_{I})n(t,x-x_{I}) dx\\
&&+\frac{1}{2}\sum_{k=1}^{I-2}N_k(t)\int_{x_{I-1}}^{x_I}\lambda_I^-(x)K(x-x_k, x_k)n(t,x-x_k) dx  \\
&& + \frac{1}{2} N_{I-1}(t)\int_{x_{I-1}}^{x_I}\lambda_I^-(x)K(x-x_{I-1}, x_{I-1})n(t,x-x_{I-1}) dx + {\cal O}(\Delta x^2)  + \mathfrak{J}_I(f).
\end{eqnarray*}

The first, second, third and fourth term on the right hand side can be solved similar to $I_1$, $I_2$, $I_4$ and $I_5$ respectively. Thus, by substituting the value of $\mathfrak{J}_I(f)$ from Lemma \ref{consistency lem}, we have
%
\begin{eqnarray}\label{birth term for I}
B_I= \hat{B}_I + \frac{f(x_I)}{4}[\Delta x_I-\Delta x_{I-1}]+ {\cal O}(\Delta x^2).
\end{eqnarray}

\subsection{Discretization error in the death term}\label{Order of the death term}
From equation (\ref{E:D_i}), the integrated death term can be written as follows
\begin{eqnarray*}
D_i = \int_{x_{i-1/2}}^{x_{i+1/2}}\sum_{j=1}^{I}\int_{x_{j-1/2}}^{x_{j+1/2}}K(x, y)n(t,y)n(t,x)dy dx.
\end{eqnarray*}
The application of the midpoint rule to the outer and inner integrals gives us
%
\begin{equation}\label{death term for i}
D_i = N_i(t)\sum_{j=1}^{I}K(x_i, x_j)N_j(t) + {\cal O}(\Delta x^3)=\hat{D}_i + {\cal O}(\Delta x^3).
\end{equation}
\subsection{Summary of all terms}\label{summary of all terms}
Finally from the equations (\ref{final birth term for i}), (\ref{birth term for 1}), (\ref{birth term for I}) and
(\ref{death term for i}), we can summarize the spatial truncation error $\sigma _i(t)=B_i-D_i-(\hat{B}_i-\hat{D}_i)$
as follows
\begin{equation}\label{E:ConsistencyOrder_1}
\sigma_1(t)={\cal O}(\Delta x^2),
\end{equation}
\vspace{-.60cm}
\begin{eqnarray}\label{E:ConsistencyOrder_i}
\hspace{.6in}\sigma_i(t)= && E+ E'+ \frac{f(x_i)}{2}\bigg[\Delta x_i-\bigg(\frac{\Delta x_{i-1}+\Delta x_{i+1}}{2}\bigg)\bigg]\\
&&-\frac{f'(x_i)}{12}\bigg[(\Delta x_{i+1}-\Delta x_{i-1})\bigg\{\Delta x_i+\bigg(\frac{\Delta x_{i-1}+\Delta x_{i+1}}{2}\bigg)\bigg\}\bigg]\nonumber\\
&&\nonumber+{\cal O}(\Delta x^3),\ \ i=2, \ldots,I-1,
\end{eqnarray}
\vspace{-.60cm}
\begin{equation}\label{E:ConsistencyOrder_I}
\sigma_I(t)= \frac{f(x_I)}{4}[\Delta x_I-\Delta x_{I-1}]+ {\cal
O}(\Delta x^2),
\end{equation}
where $E$ and $E'$ are defined in (\ref{birth term 2}).

\subsection{Meshes}\label{Meshes}
Now let us consider the following five different types of meshes to
calculate the order of local discretization error. Most of the grids are the same which are used for solving breakage PBE in J. Kumar and Warnecke \cite{JKUMARNM:2008}.
\subsubsection{Uniform mesh} Let us begin with the case of uniform mesh i.e.\ $\Delta x_i = \Delta x \ \ \mbox{and}\ \ x_i=(i-1/2)\Delta x\ \ \mbox{for any}\ \ i=1, \ldots, I$.
To estimate $\sigma_i(t)$, first we have to evaluate the order of $E$ and $E'$. In case of such uniform grids, we have
$$x_i-x_j= x_{i-j+1/2},\hspace{.1in} x_{i+1}-x_j= x_{i-j+3/2},\hspace{.1in} \mbox{and} \hspace{.1in}x_{i-1}-x_j= x_{i-j-1/2}.$$
By using the definition of indices in (\ref{defi of index fpt}), we have
$$x_i-x_j= x_{i-j+1/2} \in \Lambda _{l_{i,j}}.$$
This implies that
\vspace{-.70cm}
$$x_i-x_j= x_{i-j+1/2}=x_{l_{i,j}+ \frac{1}{2}\gamma_{i,j}}.$$
Similarly, we can easily obtain
$$x_{i+1}-x_j= x_{i-j+3/2}=x_{l_{i+1,j}+ \frac{1}{2}\gamma_{i+1,j}},$$
and
\vspace{-.70cm}
$$x_{i-1}-x_j=x_{i-j-1/2}=x_{l_{i-1,j}+ \frac{1}{2}\gamma_{i-1,j}}.$$
Therefore, from (\ref{birth term 2}) we have $E=0$. For uniform
grids, we can observe from above that $x_{i-1}-x_j$, $x_i-x_j$,
and $x_{i+1}-x_j$ are the right end boundaries of adjacent cells,
i.e.\ $l_{i-1,j}=(i-j-1)$th, $l_{i,j}=(i-j)$th and
$l_{i+1,j}=(i-j+1)$th cells, respectively. Here
$\gamma_{i-1,j}=\gamma_{i,j}=\gamma_{i+1,j}=1$. Thus, by
substituting the values of all these indices in $E'$ defined in
(\ref{birth term 2}) and using the Taylor series expansion we
obtain $E'={\cal O}(\Delta x^3)$. Finally, from the equations
(\ref{E:ConsistencyOrder_1})-(\ref{E:ConsistencyOrder_I}) we
estimate
\begin{eqnarray*}
\sigma_i(t)=
\left\{
    \begin{array}{ll}
        {\cal O}(\Delta x^2)\ \, & \ i=1, I, \\
{\cal O}(\Delta x^3)\ \, & \ i=2, \ldots, I-1.
    \end{array}
\right.
\end{eqnarray*}

The order of consistency is given by
\begin{eqnarray*}
 \|\mathbf{\sigma}(t)\|=|\sigma_1(t)|+\sum_{i=2}^{I-1}|\sigma_i(t)|+|\sigma_I(t)|
={\cal O}(\Delta x^2).
\end{eqnarray*}
Thus the method is second order consistent on a uniform mesh.

\subsubsection{Non-uniform smooth mesh $(x_{i+1/2}=r x_{i-1/2}$, $r>1$, $i=1, \ldots, I$)}\label{smooth}

In general, non-uniform smooth grids can be obtained by applying
some smooth transformation to uniform grids. Let us consider a
variable $\xi$ with uniform grids and a smooth transformation $x =
g(\xi)$ such that $x_{i\pm 1/2} = g(\xi_{i\pm 1/2})$ for any $i=1,
\ldots, I$ to get non-uniform smooth mesh.  The scheme again gives
us second order accuracy. Let $h$ be the uniform mesh width in the variable $\xi$. In
case of smooth grids, Taylor series expansions in smooth
transformations give
\vspace{-.40cm}
\begin{eqnarray*}
 \Delta x_{i}=x_{i+1/2}-x_{i-1/2}=g(\xi_i + h/2)-g(\xi_i - h/2)=h g'(\xi_i)+\frac{h^3}{24}g''(\xi_i)+{\cal O}(h^4).
\end{eqnarray*}
Hence, by evaluating $\Delta x_{i-1}$ and $\Delta x_{i+1}$ in a similar way, we obtain for any $i=2, \ldots, I$
 \begin{eqnarray*}
  \Delta x_{i}-\Delta x_{i-1}= {\cal O}(h^2), \hspace{.1in} 2\Delta x_{i}-(\Delta x_{i-1} +\Delta x_{i+1})\ \ \mbox{and}\ \ {\Delta x_{i}}^2-\Delta x_{i-1}\Delta
x_{i+1}={\cal O}(h^3).
 \end{eqnarray*}

 The above identities help us to simplify the equations
(\ref{E:ConsistencyOrder_1})-(\ref{E:ConsistencyOrder_I}) and give
\begin{equation}\label{lde smooth}
\sigma_1(t) = {\cal O}(h^2),\ \ \sigma_I(t) = {\cal O}(h^2) \quad
\mbox{and} \quad \sigma_i(t) = E + E' + {\cal O}(h^3).
\end{equation}
To find the order of $\sigma _i$, we need to evaluate $E$ and $E'$
separately. Let us first consider $E$ from (\ref{birth term 2})
and set $f(t, x, y):= K(x,y)n(t,x)$. Then the applications of the
left rectangle rule to the integrals in first and third terms as
well as the right rectangle rule to the integrals in second and
fourth terms estimate
\begin{eqnarray}\label{comb E1,E2}
 \hspace{.6in}E=&&\frac{1}{4}\sum_{j=1}^{i-1}N_j(t) \frac{(x_{l_{i,j}+ \frac{1}{2}\gamma_{i,j}}-x_i+x_j)^2}{x_i-x_{i+1}}f(t,x_i-x_j,x_j)\\
&&-\frac{1}{4}\sum_{j=1}^{i}N_j(t)\frac{(x_{l_{i+1,j}+ \frac{1}{2}\gamma_{i+1,j}}-x_{i+1}+x_j)^2}{x_i-x_{i+1}}f(t,x_{i+1}-x_j,x_j)\nonumber\\
&&+\frac{1}{4}\sum_{j=1}^{i-2}N_j(t)\frac{(x_{l_{i-1,j}+ \frac{1}{2}\gamma_{i-1,j}}-x_{i-1}+x_j)^2}{x_i-x_{i-1}}f(t,x_{i-1}-x_j,x_j)\nonumber\\
&&-\frac{1}{4}\sum_{j=1}^{i-1}N_j(t) \frac{(x_{l_{i,j}+
\frac{1}{2}\gamma_{i,j}}-x_i+x_j)^2}{x_i-x_{i-1}}f(t,x_i-x_j,x_j)\nonumber\\
&&+\frac{1}{6}\sum_{j=1}^{i-1}N_j(t) \frac{(x_{l_{i,j}+ \frac{1}{2}\gamma_{i,j}}-x_i+x_j)^3}{x_i-x_{i+1}}f_{x'}(t,x_i-x_j,x_j)\nonumber\\
&&-\frac{1}{6}\sum_{j=1}^{i}N_j(t)\frac{(x_{l_{i+1,j}+ \frac{1}{2}\gamma_{i+1,j}}-x_{i+1}+x_j)^3}{x_i-x_{i+1}}f_{x'}(t,x_{i+1}-x_j,x_j)\nonumber\\
&&+\frac{1}{6}\sum_{j=1}^{i-2}N_j(t)\frac{(x_{l_{i-1,j}+ \frac{1}{2}\gamma_{i-1,j}}-x_{i-1}+x_j)^3}{x_i-x_{i-1}}f_{x'}(t,x_{i-1}-x_j,x_j)\nonumber\\
&&\nonumber-\frac{1}{6}\sum_{j=1}^{i-1}N_j(t) \frac{(x_{l_{i,j}+
\frac{1}{2}\gamma_{i,j}}-x_i+x_j)^3}{x_i-x_{i-1}}f_{x'}(t,x_i-x_j,x_j)+{\cal
O}(\Delta x^3).
\end{eqnarray}
Let us denote the combination of first four and last four sums on
the right-hand side in (\ref{comb E1,E2}) by $E_1$ and $E_2$
respectively. We shall solve them separately. First we take $E_1$
and approximate $f$ at $(t,x_i-x_j,x_j)$ by $f$ expanded around
$(t,x_{i+1}-x_j,x_j)$ in first and fourth summations to get
\begin{eqnarray*}
 E_1=\frac{1}{4}\sum_{j=1}^{i-1}N_j&&(t) \frac{(x_{l_{i,j}+ \frac{1}{2}\gamma_{i,j}}-x_i+x_j)^2}{x_i-x_{i+1}}\\
 &&\times\{f(t,x_{i+1}-x_j,x_j)+ (x_i-x_{i+1})f_{x'}(t,x_{i+1}-x_j,x_j)\}\\
-\frac{1}{4}\sum_{j=1}^{i}&&N_j(t)\frac{(x_{l_{i+1,j}+ \frac{1}{2}\gamma_{i+1,j}}-x_{i+1}+x_j)^2}{x_i-x_{i+1}}f(t,x_{i+1}-x_j,x_j)\\
+\frac{1}{4}\sum_{j=1}^{i-2}&&N_j(t)\frac{(x_{l_{i-1,j}+ \frac{1}{2}\gamma_{i-1,j}}-x_{i-1}+x_j)^2}{x_i-x_{i-1}}f(t,x_{i-1}-x_j,x_j)\\
-\frac{1}{4}\sum_{j=1}^{i-1}&&N_j(t)
\frac{(x_{l_{i,j}+\frac{1}{2}\gamma_{i,j}}-x_i+x_j)^2}{x_i-x_{i-1}}\\
&&\times\{f(t,x_{i-1}-x_j,x_j)+
(x_i-x_{i-1})f_{x'}(t,x_{i-1}-x_j,x_j)\}+ {\cal O}(\Delta x^3).
\end{eqnarray*}

Again approximating $f_{x'}$ at point $(t,x_{i+1}-x_{j}, x_j)$ by
$f_{x'}$ evaluated at point $(t,x_{i-1}-x_j,x_j)$ in the second
part of the first summation. Then we replace $j$ by $j+1$ and
$j-1$ respectively in second and third term summations. Further we
can use the relationship $N_j(t)=n(t,x_j)\Delta x_j+ {\cal
O}(\Delta x^3)$ for the midpoint rule to obtain
\begin{eqnarray*}
 E_1=&&\frac{1}{4}n(t,x_{1})\Delta x_{1} \frac{(x_{l_{i+1,1}+ \frac{1}{2}\gamma_{i+1,1}}-x_{i+1}+x_{1})^2}{x_{i+1}-x_{i}}f(t,x_{i+1}-x_{1},x_{1})\\
 &&+\frac{1}{4}\sum_{j=1}^{i-1}n(t,x_{j+1})\Delta x_{j+1} \frac{(x_{l_{i+1,j+1}+ \frac{1}{2}\gamma_{i+1,j+1}}-x_{i+1}+x_{j+1})^2}{x_{i+1}-x_{i}}f(t,x_{i+1}-x_{j+1},x_{j+1})\\
&&-\frac{1}{4}\sum_{j=1}^{i-1}n(t,x_j)\Delta x_j\frac{(x_{l_{i,j}+ \frac{1}{2}\gamma_{i,j}}-x_i+x_j)^2}{x_{i+1}-x_{i}}f(t,x_{i+1}-x_j,x_j)\\
&&+\frac{1}{4}\sum_{j=2}^{i-1}n(t,x_{j-1})\Delta x_{j-1}\frac{(x_{l_{i-1,j-1}+ \frac{1}{2}\gamma_{i-1,j-1}}-x_{i-1}+x_{j-1})^2}{x_i-x_{i-1}}f(t,x_{i-1}-x_{j-1},x_{j-1})\\
&&-\frac{1}{4}\sum_{j=2}^{i-1}n(t,x_j)\Delta x_j \frac{(x_{l_{i,j}+
\frac{1}{2}\gamma_{i,j}}-x_i+x_j)^2}{x_i-x_{i-1}}f(t,x_{i-1}-x_j,x_j)\\
&&-\frac{1}{4}n(t,x_1)\Delta x_1 \frac{(x_{l_{i,1}+
\frac{1}{2}\gamma_{i,1}}-x_i+x_1)^2}{x_i-x_{i-1}}f(t,x_{i-1}-x_1,x_1)+{\cal O}(\Delta x^3).
\end{eqnarray*}
Approximating the functions $x'\mapsto f(t,x_{i\pm
1}-x',x')n(t,x')$ at point $x_j$ by $f(t,x_{i\pm 1}-x',x')n(t,x')$
evaluated at point $x'=x_{j\pm 1}$ of the third and fifth terms on
the right hand side respectively. Then we approximate $f$ at
$(t,x_{i-1}-x_1,x_1)$ by $f$ expanded around $(t,x_{i+1}-x_1,x_1)$
in the last term on the right-hand side and rearrange the whole expression for
$E_1$ as follows
\begin{eqnarray*}
E_1=\frac{1}{4}n(t,x_{1})\Delta x_{1} &&\frac{(x_{l_{i+1,1}+ \frac{1}{2}\gamma_{i+1,1}}-x_{i+1}+x_{1})^2}{x_{i+1}-x_{i}}f(t,x_{i+1}-x_{1},x_{1})\\
+\frac{1}{4}\sum_{j=1}^{i-1}&&n(t,x_{j+1})(\Delta x_{j+1}-\Delta x_{j})\\
&&\times \frac{(x_{l_{i+1,j+1}+ \frac{1}{2}\gamma_{i+1,j+1}}-x_{i+1}+x_{j+1})^2}{x_{i+1}-x_{i}}f(t,x_{i+1}-x_{j+1},x_{j+1})\\
+\frac{1}{4}\sum_{j=1}^{i-1}&&n(t,x_{j+1})\Delta
x_{j}\frac{f(t,x_{i+1}-x_{j+1},x_{j+1})}{x_{i+1}-x_{i}}\nonumber\\
&&\times \bigg\{(x_{l_{i+1,j+1}+
\frac{1}{2}\gamma_{i+1,j+1}}-x_{i+1}+x_{j+1})^2-(x_{l_{i,j}+
\frac{1}{2}\gamma_{i,j}}-x_i+x_j)^2\bigg\}\nonumber\\
+\frac{1}{4}\sum_{j=1}^{i-1}&&\Delta x_j \frac{\partial}{\partial
x'}\{n(t,x_{j+1})f(t,x_{i+1}-x_{j+1},x_{j+1})\}\\
&&\times\frac{(x_{j+1}-x_{j})}{x_{i+1}-x_{i}}(x_{l_{i,j}+
\frac{1}{2}\gamma_{i,j}}-x_i+x_j)^2\nonumber\\
-\frac{1}{4}\sum_{j=2}^{i-1}&&n(t,x_{j-1})(\Delta x_{j}-\Delta x_{j-1})\\
&&\times \frac{(x_{l_{i-1,j-1}+ \frac{1}{2}\gamma_{i-1,j-1}}-x_{i-1}+x_{j-1})^2}{x_{i}-x_{i-1}}f(t,x_{i-1}-x_{j-1},x_{j-1})\\
-\frac{1}{4}\sum_{j=2}^{i-1}&&n(t,x_{j-1})\Delta
x_{j}\frac{f(t,x_{i-1}-x_{j-1},x_{j-1})}{x_{i}-x_{i-1}}\nonumber\\
&&\times \bigg\{(x_{l_{i,j}+
\frac{1}{2}\gamma_{i,j}}-x_{i}+x_{j})^2-(x_{l_{i-1,j-1}+
\frac{1}{2}\gamma_{i-1,j-1}}-x_{i-1}+x_{j-1})^2\bigg\}\nonumber\\
-\frac{1}{4}\sum_{j=2}^{i-1}&&\Delta x_j \frac{\partial}{\partial
x'}\{n(t,x_{j-1})f(t,x_{i-1}-x_{j-1},x_{j-1})\}\\
&&\times\frac{(x_{j}-x_{j-1})}{x_{i}-x_{i-1}}(x_{l_{i,j}+
\frac{1}{2}\gamma_{i,j}}-x_i+x_j)^2 \\
-\frac{1}{4}n(t,x_1)&&\Delta x_1 \frac{(x_{l_{i,1}+
\frac{1}{2}\gamma_{i,1}}-x_i+x_1)^2}{x_i-x_{i-1}}f(t,x_{i+1}-x_1,x_1)
+{\cal O}(\Delta x^3).
\end{eqnarray*}
Let us denote the each term on the right hand side by $E_{11}, \ldots,
E_{18}$ respectively. Therefore, we have
\begin{equation}\label{Error E}
\hspace{.6in} E_1=(E_{11}-E_{18})+
(E_{12}-E_{15})+(E_{13}-E_{16})+(E_{14}-E_{17}) + {\cal O}(\Delta
x^3).
\end{equation}
In order to solve the error $E_1$, we need to evaluate differences
of the above mentioned terms separately. First of all, we are
interested in $E_{12}-E_{15}$. In case of smooth grids, we have as
before $\Delta x_j-\Delta x_{j-1}= {\cal O}(\Delta x^2)$ which
clearly shows that $E_{15}={\cal O}(\Delta x^2)$. Then Taylor's
series expansion gives us a second order approximation of $E_{15}$
as
\begin{eqnarray*}
 E_{15}=\frac{1}{4}\sum_{j=2}^{i-1}&&n(x_{j+1})f(x_{i+1}-x_{j+1},x_{j+1})(\Delta x_{j}-\Delta
 x_{j-1})\\
&&\times\frac{(x_{l_{i-1,j-1}+\frac{1}{2}\gamma_{i-1,j-1}}-x_{i-1}+x_{j-1})^2}{x_{i}-x_{i-1}}+{\cal
O}(\Delta x^3).
\end{eqnarray*}
Without loss of generality the summation in $E_{12}$ can be
started from $j=2$ since the term we are dropping from the summation
is third order accurate. Then we can write
\begin{eqnarray}\label{E125}
 \hspace{.6in}E_{12}-E_{15}=\frac{1}{2}&&\sum_{j=2}^{i-1}n(x_{j+1})f(x_{i+1}-x_{j+1},x_{j+1})\\
 \times &&\bigg\{\frac{(\Delta x_{j+1}-\Delta x_{j})}{\Delta x_{i}+\Delta x_{i+1}}(x_{l_{i+1,j+1}
+ \frac{1}{2}\gamma_{i+1,j+1}}-x_{i+1}+x_{j+1})^2\nonumber\\
&&- \frac{(\Delta x_{j}-\Delta x_{j-1})}{\Delta x_{i}+\Delta x_{i-1}}(x_{l_{i-1,j-1}+\frac{1}{2}\gamma_{i-1,j-1}}-x_{i-1}+x_{j-1})^2\bigg\}\nonumber\\
&&\nonumber+{\cal O}(\Delta x^3).
\end{eqnarray}
By applying Taylor series expansion in the smooth transformation,
we obtain
\begin{equation}\label{i-1,j+1}
 (\Delta x_{i}+\Delta x_{i-1})(\Delta x_{j+1}-\Delta x_{j})=2h^3 g'(\xi_i)g''(\xi_j)+{\cal O}(h^4),
\end{equation}
and
\vspace{-.25cm}
\begin{equation}\label{i+1,j-1}
 (\Delta x_{i}+\Delta x_{i+1})(\Delta x_{j}-\Delta x_{j-1})=2h^3 g'(\xi_i)g''(\xi_j)+{\cal O}(h^4).
\end{equation}
Here we consider a particular type of non-uniform smooth grids
i.e.\ $x_{i+1/2} = r x_{i-1/2}, r>1, i = 1, \ldots I$. Such
grids are called \textbf{geometric grids}. These grids can be
obtained by applying a smooth transformation as
\textbf{exponential function} on uniform grids. Here we have
$x_{i+1/2} = \exp(\xi_{i+1/2}) = \exp(h +\xi_{i-1/2}) = \exp(h)
\exp(\xi_{i-1/2}) =
\exp(h) x_{i-1/2} =: r x_{i-1/2}, \, r>1$.\\
Let us consider $\xi_{11}$, $\xi_{12}$, $\xi_{21}$, $\xi_{22}$, $\xi_{31}$ and $\xi_{32}$ are corresponding
points on uniform mesh for $x_{l_{i+1,j+1}+ \frac{1}{2}\gamma_{i+1,j+1}}$, $x_{i+1}-x_{j+1}$,
$x_{l_{i,j}+ \frac{1}{2}\gamma_{i,j}}$, $x_{i}-x_{j}$, $x_{l_{i-1,j-1}+ \frac{1}{2}\gamma_{i-1,j-1}}$ and $x_{i-1}-x_{j-1}$, respectively. These are defined as follows.
$$\xi_{11}=\mbox{ln}(x_{l_{i+1,j+1}+\frac{1}{2}\gamma_{i+1,j+1}}), \ldots, \xi_{32}=\mbox{ln}(x_{i-1}-x_{j-1}).$$
By the definition of the indices in (\ref{defi of index fpt}), we know
$$x_{i+1}-x_{j+1} \in \Lambda_{l_{i+1,j+1}}, \ \ x_{i}-x_{j} \in \Lambda_{l_{i,j}}\ \ \mbox{and} \ \ x_{i-1}-x_{j-1} \in \Lambda_{l_{i-1,j-1}}.$$
For geometric grids, we have
$$x_{i+1}-x_{j+1}=r(x_{i}-x_{j})=r^2(x_{i-1}-x_{j-1}).$$
Therefore, we have
\vspace{-.40cm}
$$l_{i+1,j+1}=l_{i,j}+1=l_{i-1,j-1}+2.$$
Further, in case of geometric grids, we have
$$\gamma_{i+1,j+1}=\gamma_{i,j}=\gamma_{i-1,j-1}.$$
Let us consider
\begin{eqnarray*}
 h_1=\xi_{11}-\xi_{12}=&&\mbox{ln}(x_{l_{i+1,j+1}+\frac{1}{2}\gamma_{i+1,j+1}})-\mbox{ln}(x_{i+1}-x_{j+1})\\
&&=\mbox{ln}\bigg(\frac{x_{l_{i+1,j+1}+\frac{1}{2}\gamma_{i+1,j+1}}}{x_{i+1}-x_{j+1}}\bigg)=\mbox{ln}\bigg(\frac{x_{l_{i,j}+\frac{1}{2}\gamma_{i,j}}}{x_{i}-x_{j}}\bigg)=\xi_{21}-\xi_{22}\\
&&=\mbox{ln}\bigg(\frac{x_{l_{i-1,j-1}+\frac{1}{2}\gamma_{i-1,j-1}}}{x_{i-1}-x_{j-1}}\bigg)=\xi_{31}-\xi_{32}.
\end{eqnarray*}

Similarly, we can estimate
\begin{eqnarray*}
 \xi_{12}-\xi_{32}= \mbox{ln}\bigg(\frac{x_{i+1}-x_{j+1}}{x_{i-1}-x_{j-1}}\bigg)=\mbox{ln}(r^2)=2h.
\end{eqnarray*}
By the application of smooth transformation we can write
\begin{eqnarray}\label{smooth12}
\hspace{.6in} [x_{l_{i+1,j+1}+\frac{1}{2}\gamma_{i+1,j+1}}&&-(x_{i+1}-x_{j+1})]^2\\
 &&\nonumber=[g(\xi_{11})-g(\xi_{12})]^2=h_1^2\{g'(\xi_{12})\}^2 + {\cal O}(h^3),
\end{eqnarray}
and
\begin{eqnarray}\label{smooth32}
\hspace{.6in} [x_{l_{i-1,j-1}+\frac{1}{2}\gamma_{i-1,j-1}}&&-(x_{i-1}-x_{j-1})]^2\\
 &&\nonumber=[g(\xi_{31})-g(\xi_{32})]^2=h_1^2\{g'(\xi_{32})\}^2 + {\cal O}(h^3).
\end{eqnarray}
 Collecting (\ref{i-1,j+1}), (\ref{i+1,j-1}), (\ref{smooth12}) and (\ref{smooth32}) together, we obtain
\begin{eqnarray*}
 (\Delta x_{i}+&&\Delta x_{i-1})(\Delta x_{j+1}-\Delta x_{j})(x_{l_{i+1,j+1}+ \frac{1}{2}\gamma_{i+1,j+1}}-x_{i+1}+x_{j+1})^2\\
&&- (\Delta x_{i}+\Delta x_{i+1})(\Delta x_{j}-\Delta x_{j-1})(x_{l_{i-1,j-1}+\frac{1}{2}\gamma_{i-1,j-1}}-x_{i-1}+x_{j-1})^2\\
=&&-8 h^4 h_1^2 g'(\xi_i)g''(\xi_j)g'(\xi_{12})g''(\xi_{12})+ {\cal O}(h^6)= {\cal O}(h^6).
\end{eqnarray*}
Therefore, by substituting the preceding identity in (\ref{E125}) we estimate
\begin{equation}\label{Error1}
 E_{12}-E_{15}={\cal O}(\Delta x^3).
\end{equation}

Next, we need to compute $E_{13}-E_{16}$ defined in (\ref{Error E}).
In case of smooth grids, by using smooth transformations and Taylor's series expansions, it is easy to show that
\vspace{-.40cm}
\begin{eqnarray*}
 [x_{l_{i,j}+\frac{1}{2}\gamma_{i,j}}-(x_{i}-x_{j})]^2-[x_{l_{i-1,j-1}+\frac{1}{2}\gamma_{i-1,j-1}}-(x_{i-1}-x_{j-1})]^2={\cal O}(h^3).
\end{eqnarray*}
which implies that $E_{16}={\cal O}(\Delta x^2)$. So we can take a second order approximation of $E_{16}$ by using Taylor's series expansion as
\begin{eqnarray*}
 E_{16}= \frac{1}{2}&&\sum_{j=2}^{i-1} n(x_{j+1})\Delta
x_{j}\frac{f(x_{i+1}-x_{j+1},x_{j+1})}{\Delta x_{i}+ \Delta x_{i-1}}\nonumber\\
&&\times \bigg\{(x_{l_{i,j}+
\frac{1}{2}\gamma_{i,j}}-x_{i}+x_{j})^2-(x_{l_{i-1,j-1}+
\frac{1}{2}\gamma_{i-1,j-1}}-x_{i-1}+x_{j-1})^2\bigg\}+ {\cal O}(\Delta x^3)\nonumber\\
\end{eqnarray*}
Similarly as before, we can have
$$[x_{l_{i+1,j+1}+\frac{1}{2}\gamma_{i+1,j+1}}-(x_{i+1}-x_{j+1})]^2-[x_{l_{i,j}+\frac{1}{2}\gamma_{i,j}}-(x_{i}-x_{j})]^2={\cal O}(h^3).$$
Therefore, again without loss of generality the summation in $E_{13}$ can be started from $j=2$ since the term we are dropping from the summation is third order accurate. Let us estimate
\begin{eqnarray}\label{E136}
\hspace{.6in} E_{13}&&-E_{16}=\\
 \frac{1}{2}&&\sum_{j=2}^{i-1} n(x_{j+1})\Delta x_{j}f(x_{i+1}-x_{j+1},x_{j+1})\nonumber\\
 \times&& \bigg[\frac{\{(x_{l_{i+1,j+1}+
\frac{1}{2}\gamma_{i+1,j+1}}-x_{i+1}+x_{j+1})^2-(x_{l_{i,j}+
\frac{1}{2}\gamma_{i,j}}-x_{i}+x_{j})^2\}}{\Delta x_{i}+ \Delta x_{i+1}}\nonumber\\
&&-\frac{\{(x_{l_{i,j}+
\frac{1}{2}\gamma_{i,j}}-x_{i}+x_{j})^2-(x_{l_{i-1,j-1}+
\frac{1}{2}\gamma_{i-1,j-1}}-x_{i-1}+x_{j-1})^2\}}{\Delta x_{i}+ \Delta x_{i-1}}\bigg]\nonumber\\
+&& {\cal O}(\Delta x^3).\nonumber
\end{eqnarray}
By using smooth transformation and Taylor's series expansions, we can easily obtain
\begin{eqnarray*}
(\Delta x_{i}&&+ \Delta x_{i-1}) \{(x_{l_{i+1,j+1}+
\frac{1}{2}\gamma_{i+1,j+1}}-x_{i+1}+x_{j+1})^2-(x_{l_{i,j}+
\frac{1}{2}\gamma_{i,j}}-x_{i}+x_{j})^2\}\nonumber\\
&&-(\Delta x_{i}+ \Delta x_{i+1}) \{(x_{l_{i,j}+
\frac{1}{2}\gamma_{i,j}}-x_{i}+x_{j})^2-(x_{l_{i-1,j-1}+
\frac{1}{2}\gamma_{i-1,j-1}}-x_{i-1}+x_{j-1})^2\}\nonumber\\
=&&2h^2 h_1^2 g'(\xi_{22})g''(\xi_{22})-2h^2 h_1^2 g'(\xi_{22})g''(\xi_{22})+{\cal O}(h^5)={\cal O}(h^5).
\end{eqnarray*}
Putting the above identity in equation (\ref{E136}), we have
\begin{equation}\label{Error2}
 E_{13}-E_{16}={\cal O}(\Delta x^3).
\end{equation}
 Now we shall approach to evaluate $E_{14}-E_{17}$. Since the terms appearing in the summation $E_{14}$ are third order accurate. Therefore, the summation in $E_{14}$ can be started from $j=2$. By using Taylor's series expansion in $E_{17}$ we consider
\begin{eqnarray}\label{E147}
 \hspace{.6in}E_{14}-E_{17}=\frac{1}{4}&&\sum_{j=2}^{i-1}\Delta x_j \frac{\partial}{\partial x'}\{n(x_{j+1})f(x_{i+1}-x_{j+1},x_{j+1})\}\\
&& \times \bigg[\frac{\Delta x_{j}+\Delta x_{j+1}}{\Delta x_{i}+\Delta x_{i+1}}
-\frac{\Delta x_{j}+\Delta x_{j-1}}{\Delta x_{i}+\Delta x_{i-1}}\bigg](x_{l_{i,j}+
\frac{1}{2}\gamma_{i,j}}-x_i+x_j)^2\nonumber\\
&&\nonumber+ {\cal O}(\Delta x^3).
\end{eqnarray}
%
Analogous to equations (\ref{smooth12}) and (\ref{smooth32}), we estimate
\begin{equation}\label{lij}
(x_{l_{i,j}+\frac{1}{2}\gamma_{i,j}}-x_i+x_j)^2={\cal O}(h^2).
\end{equation}
By applying the smooth transformation and Taylor's series expansion, we have
\begin{equation}\label{1,i-1,j+1}
(\Delta x_{i}+\Delta x_{i-1})(\Delta x_{j}+\Delta x_{j+1})=4h^2 g'(\xi_i)g'(\xi_j)+{\cal O}(h^3),
\end{equation}
and
\begin{equation}\label{1,i+1,j-1}
(\Delta x_{i}+\Delta x_{i+1})(\Delta x_{j}+\Delta x_{j-1})=4h^2 g'(\xi_i)g'(\xi_j)+{\cal O}(h^3).
\end{equation}
Substituting the identities (\ref{lij}), (\ref{1,i-1,j+1}) and (\ref{1,i+1,j-1}) in equation (\ref{E147}), we obtain
\begin{equation}\label{Error3}
E_{14}-E_{17}={\cal O}(\Delta x^3).
\end{equation}
Finally, we have to estimate $E_{11}-E_{18}$ in equation (\ref{Error E}). 
Since the geometric grids considered here are monotonically increasing. So, the points $x_{i}-x_{1}$ and $x_{i+1}-x_{1}$ lie in $i$th and $(i+1)$th cells, respectively. By using the definition of the indices $l_{i,j}$ and $\gamma_{i,j}$ in (\ref{defi of index fpt}), we have $l_{i,1}=i$, $l_{i+1,1}=i+1$ and $\gamma_{i,1}=\gamma_{i+1,1}= -1$. Then the expression for $E_{11}-E_{18}$ can be written as
 \begin{eqnarray}\label{E118}
\hspace{.6in}E_{11}-E_{18}=\frac{1}{2}n(t,x_{1})\Delta x_{1}\bigg[&&\frac{(x_{i+1/2}-x_{i+1}+x_{1})^2}{\Delta x_{i}+\Delta x_{i+1}}\\
&&\nonumber-\frac{(x_{i-1/2}-x_i+x_1)^2}{\Delta x_{i}+\Delta x_{i-1}}\bigg]f(t,x_{i+1}-x_{1},x_{1}).
\end{eqnarray}
Let us use the following formula and the identities mentioned before (\ref{lde smooth}) to get
\begin{eqnarray*}
(\Delta x_{i}+\Delta x_{i-1})(&&x_{i+1/2}-x_{i+1}+x_{1})^2-(\Delta x_{i}+\Delta x_{i+1})(x_{i-1/2}-x_i+x_1)^2\\
=&&\frac{1}{4}[\Delta x_{i}(\Delta x_{i+1}+\Delta x_{i})(\Delta x_{i+1}-\Delta x_{i})-{\Delta x_{1}}^2(\Delta x_{i+1}-\Delta x_{i-1})\\
&&+(2\Delta x_{1}-\Delta x_{i+1})({\Delta x_{i}}^2-\Delta x_{i-1}\Delta x_{i+1})]\\
=&&{\cal O}(h^4).
\end{eqnarray*}
Inserting the above identity in equation (\ref{E118}), we estimate
\begin{equation}{\label{Error4}}
E_{11}-E_{18}={\cal O}(\Delta x^3).
\end{equation}
Finally, we substitute (\ref{Error1}), (\ref{Error2}), (\ref{Error3}) and (\ref{Error4}) into (\ref{Error E}) to obtain
\begin{eqnarray*}
E_1={\cal O}(\Delta x^3).
\end{eqnarray*}
Similar to $E_1$, we can evaluate $E_2$ defined in (\ref{comb E1,E2}) as
\begin{eqnarray*}
E_2={\cal O}(\Delta x^3).
\end{eqnarray*}
Therefore, substituting $E_1$ and $E_2$ in (\ref{comb E1,E2}), we obtain
 \begin{eqnarray*}
E={\cal O}(\Delta x^3).
\end{eqnarray*}
Now let us consider $E'$ from (\ref{birth term 2}) and set $f_x(t,x,y):=\frac{\partial}{\partial x}\{K(x,y)n(t,x)\}$. Then we replace $j$ by $j-1$ in the second term on the right hand side and apply Taylor series expansion about $x_j=x_{j-1}$ in $n(t,x_j)f_{x'}(t,x_k,x_{j})$ of the first term. Further we rearrange the expression for $E'$ to estimate
\begin{eqnarray*}
 E'=&&\frac{1}{12}\sum_{j=2}^{i-1}n(t,x_{j-1})(\Delta x_{j-1}-\Delta x_j)\sum_{k=l_{i,j}+ \frac{1}{2}(1+\gamma_{i,j})}^{l_{i+1,j}+ \frac{1}{2}(\gamma_{i+1,j}-1)}
\frac{{\Delta x_k}^3}{\Delta x_{i}+\Delta x_{i+1}}f_{x'}(t,x_k,x_{j-1})\\
&&-\frac{1}{12}\sum_{j=2}^{i-1}n(t,x_{j-1})\Delta x_{j-1}\sum_{k=l_{i,j}+ \frac{1}{2}(1+\gamma_{i,j})}^{l_{i+1,j}+ \frac{1}{2}(\gamma_{i+1,j}-1)}
\frac{{\Delta x_k}^3}{\Delta x_{i}+\Delta x_{i+1}}f_{x'}(t,x_k,x_{j-1})\\
&&+\frac{1}{12}\sum_{j=2}^{i-1}n(t,x_{j-1})\Delta x_{j-1}\sum_{k=l_{i-1,j-1}+ \frac{1}{2}(1+\gamma_{i-1,j-1})}^{l_{i,j-1}+ \frac{1}{2}(\gamma_{i,j-1}-1)}
\frac{{\Delta x_k}^3}{\Delta x_{i}+\Delta x_{i-1}}f_{x'}(t,x_k,x_{j-1})\\
&&+{\cal O}(\Delta x^3).
\end{eqnarray*}
For smooth grids, we have $\Delta x_{j-1}- \Delta x_j={\cal O}(\Delta x^2)$. Therefore, the first term on the right hand
side of the above equation is of third order. By using the Taylor series expansion about $x_k=x_j$ in $f_{x'}(t,x_k,x_{j-1})$ in the
remaining terms on the right hand side, the above equation can be further rewritten as
\begin{eqnarray}\label{E'}
 \hspace{.6in}E'=\frac{1}{12}\sum_{j=2}^{i-1}&&n(t,x_{j-1})\Delta x_{j-1}\bigg[\sum_{k=l_{i-1,j-1}+ \frac{1}{2}(1+\gamma_{i-1,j-1})}^{l_{i,j-1}+ \frac{1}{2}(\gamma_{i,j-1}-1)}
\frac{{\Delta x_k}^3}{\Delta x_{i}+\Delta x_{i-1}}\\
&&\nonumber-\sum_{k=l_{i,j}+ \frac{1}{2}(1+\gamma_{i,j})}^{l_{i+1,j}+ \frac{1}{2}(\gamma_{i+1,j}-1)}
\frac{{\Delta x_k}^3}{\Delta x_{i}+\Delta x_{i+1}}\bigg]f_{x'}(t,x_j,x_{j-1})\nonumber+{\cal O}(\Delta x^3).
\end{eqnarray}
For non-uniform smooth grids of the type $x_{i+1/2}=r x_{i-1/2}$, $r>1$, $i=1, \ldots, I$, we have
$$x_{i+1}-x_j=r(x_{i}-x_{j-1})\ \ \mbox{and}\ \ x_{i+1}-x_j \in \Lambda_{l_{i+1,j}},\ \ x_{i}-x_{j-1}\in \Lambda_{l_{i,j-1}}.$$
Therefore, we obtain $l_{i+1,j}=l_{i,j-1}+1$. Similarly, we can have $l_{i,j}=l_{i-1,j-1}+1$.
In case of such smooth grids, it can easily be seen that $\gamma_{i-1,j-1}=\gamma_{i,j}$ and $\gamma_{i,j-1}=\gamma_{i+1,j}$.\\
Let $k_1:=l_{i-1,j-1}+\frac{1}{2}(1+\gamma_{i-1,j-1})$, $k_2:=l_{i,j}+\frac{1}{2}(1+\gamma_{i,j})=k_1+1$, \ldots, $k_p:=l_{i,j-1}+\frac{1}{2}(\gamma_{i,j-1}-1)=k_{p-1}+1$,
$K_{p+1}:=l_{i+1,j}+\frac{1}{2}(\gamma_{i+1,j}-1)=k_{p}+1$. Then we can write
\begin{eqnarray}\label{E'1}
\hspace{.6in} \sum_{k=l_{i-1,j-1}+ \frac{1}{2}(1+\gamma_{i-1,j-1})}^{l_{i,j-1}+ \frac{1}{2}(\gamma_{i,j-1}-1)}
\frac{{\Delta x_k}^3}{\Delta x_{i}+\Delta x_{i-1}}&&
-\sum_{k=l_{i,j}+ \frac{1}{2}(1+\gamma_{i,j})}^{l_{i+1,j}+ \frac{1}{2}(\gamma_{i+1,j}-1)}
\frac{{\Delta x_k}^3}{\Delta x_{i}+\Delta x_{i+1}}\\
=\bigg[\frac{{\Delta x_{k_1}}^3}{\Delta x_{i}+\Delta x_{i-1}}&&-\frac{{\Delta x_{k_2}}^3}{\Delta x_{i}+\Delta x_{i+1}}\bigg]+\ldots\nonumber\\
&&+\bigg[\frac{{\Delta x_{k_p}}^3}{\Delta x_{i}+\Delta x_{i-1}}-\frac{{\Delta x_{k_{p+1}}}^3}{\Delta x_{i}+\Delta x_{i+1}}\bigg]\nonumber\\
=\sum_{m=1}^{p}\bigg[\frac{{\Delta x_{k_m}}^3}{\Delta x_{i}+\Delta x_{i-1}}&&-\frac{{\Delta x_{k_{m+1}}}^3}{\Delta x_{i}+\Delta x_{i+1}}\bigg]\nonumber
\end{eqnarray}
where $k_{m+1}=k_m+1$.
Let $\xi_{m1}$, $\xi_{m2}$, $\xi_{m3}$ are the corresponding points on the uniform mesh for $x_{k_m+3/2}$, $x_{k_m +1/2}$ and
$x_{k_m-1/2}$, respectively.
Since
\begin{eqnarray*}
 \xi_{m1}-\xi_{m2}=\mbox{ln}\bigg(\frac{x_{k_m+3/2}}{x_{k_m +1/2}}\bigg)=\mbox{ln}(r)=h
=\mbox{ln}\bigg(\frac{x_{k_m+1/2}}{x_{k_m -1/2}}\bigg)=\xi_{m2}-\xi_{m3}.
\end{eqnarray*}
Then by using smooth transformations and Taylor's series expansions, we can have
\begin{equation}\label{k,k+1}
 (\Delta x_{i}+\Delta x_{i+1})\Delta x_{k_m}^3-(\Delta x_{i}+\Delta x_{i-1})\Delta x_{k_{m+1}}^3={\cal O}(h^5).
\end{equation}
Therefore, by using (\ref{k,k+1}) and (\ref{E'1}) in equation (\ref{E'}), we estimate
\begin{eqnarray*}
 E'={\cal O}(\Delta x^3).
\end{eqnarray*}
Thus, we insert the value of $E$ and $E'$ in equation (\ref{lde smooth}) to obtain
\begin{eqnarray*}
 \sigma_i={\cal O}(h^3).
\end{eqnarray*}

Finally, analogous to the uniform mesh, the technique is second order consistent.
 \subsubsection{Locally uniform mesh}

  Let us explain an example of a locally uniform mesh. First, the
 computational domain is split into many finite sub-domains and each sub-domain is further split into an equal size mesh.
 This gives us a locally uniform mesh. It is not easy to analyze the order of consistency on locally uniform mesh.
 So we calculate it later numerically. In this case, the scheme gives only first order consistency.

\subsubsection{Oscillatory mesh}
A mesh is called oscillatory mesh, if for any $r \neq 1 > 0$, we have
\begin{eqnarray*}
\Delta x_{i+1}:=
\left\{
    \begin{array}{ll}
        r\Delta x_{i} & \mbox{if } \ i\ \ \mbox{is odd}, \\
        \frac{1}{r}\Delta x_{i} & \mbox{if } \ i\ \ \mbox{is even}.
    \end{array}
\right.
\end{eqnarray*}

Since there is no cancellation in the leading error terms of equations (\ref{E:ConsistencyOrder_1})-(\ref{E:ConsistencyOrder_I}) as well as $E={\cal O}(\Delta x)$ and $E'={\cal O}(\Delta x^2)$  we have
\begin{eqnarray*}
\sigma_i(t)={\cal O}(\Delta x),\ \ i=1, \ldots, I, \ \ \mbox{and}\ \
\|\mathbf{\sigma}(t)\| = {\cal O}(1).
\end{eqnarray*}
Therefore the fixed pivot method is unfortunately inconsistent on oscillatory meshes.
\subsubsection{Non-uniform random mesh}
Finally the scheme is examined on non-uniform random grids. Similar to the case of oscillatory mesh, we have $\|\mathbf{\sigma}(t)\| = {\cal O}(1)$.
Thus the method is again inconsistent on non-uniform random meshes.

\section{Lipschitz conditions on $\hat{\mathbf{B}}({\mathbf{N}}(t))$ and $\hat{\mathbf{D}}({\mathbf{N}}(t))$}\label{convergence fpt}

Let us consider the birth term for $0 \leq t \leq T$ and for all $\mathbf{N}$, $\hat{\mathbf{N}}\in \mathbb{R}^I$
\begin{eqnarray*}
 \|\hat{\mathbf{B}}({\mathbf{N}}(t))-\hat{\mathbf{B}}({\hat{\mathbf{N}}}(t))\|= \sum_{i=1}^{I}|\hat{B}_i(\mathbf{N}(t))-\hat{B}_i(\hat{\mathbf{N}}(t))|.
\end{eqnarray*}

From (\ref{cond on beta}), there exists a $L>0$ such that $K(x,y) \leq L$ for all $x,y \in ]0,x_{\mbox{max}}]$. Then by using this upper bound $L$ and $0\leq \lambda_i^{\pm}(x) \leq 1$ from the definition in (\ref{values of lambda fpt}), we obtain from (\ref{resultant system})
\begin{eqnarray*}
\|\hat{\mathbf{B}}({\mathbf{N}}(t))-\hat{\mathbf{B}}({\hat{\mathbf{N}}}(t))\| \leq &&\frac{1}{2}L\sum_{i=1}^{I}  \sum_{j=1}^{i} \sum_{x_i\leq x_j+x_k < x_{i+1}} |N_j(t) N_k(t)-\hat{N}_j(t) \hat{N}_k(t)|\\
&& + \frac{1}{2}L\sum_{i=1}^{I}  \sum_{j=1}^{i-1} \sum_{x_{i-1}\leq x_j+x_k < x_{i}} |N_j(t) N_k(t)-\hat{N}_j(t) \hat{N}_k(t)|\\
\leq && L\sum_{j=1}^{I}\sum_{k=1}^{I}|N_j(t) N_k(t)-\hat{N}_j(t)\hat{N}_k(t)|.
\end{eqnarray*}

Now we enjoy a useful equality $N_j(t) N_k(t)-\hat{N}_j(t)\hat{N}_k(t) = \frac{1}{2}[(N_j(t)+\hat{N}_j(t))(N_k(t)-\hat{N}_k(t)) +(N_j(t)-\hat{N}_j(t))(N_k(t)+\hat{N}_k(t))]$ to get

\begin{eqnarray}\label{1st lipschitz condition for B fpt}
\hspace{.6in}\|\hat{\mathbf{B}}({\mathbf{N}}(t))-\hat{\mathbf{B}}({\hat{\mathbf{N}}}(t))\| \leq \frac{1}{2} L\sum_{j=1}^{I}\sum_{k=1}^{I}\bigg[&&|(N_j(t)+\hat{N}_j(t))| |(N_k(t)-\hat{N}_k(t))| \\
 &&\nonumber+ |(N_j(t)-\hat{N}_j(t))| |(N_k(t)+\hat{N}_k(t))|\bigg].
\end{eqnarray}

It can be easily shown that the total number of particles decreases in a coagulation process, i.e.\ $\sum_{j=1}^{I}N_j(t) \leq N_T^0:= \mbox{Total number of particles which are taken initially}$. Therefore, the equation (\ref{1st lipschitz condition for B fpt}) can be rewritten as

\begin{eqnarray}\label{lipschitz condition for B fpt}
\hspace{.6in}\|\hat{\mathbf{B}}({\mathbf{N}}(t))-\hat{\mathbf{B}}({\hat{\mathbf{N}}}(t))\| && \leq  N_T^0 L\bigg[ \sum_{k=1}^{I}|(N_k(t)-\hat{N}_k(t))|+ \sum_{j=1}^{I}|(N_j(t)-\hat{N}_j(t))|\bigg]\\
&& \nonumber\leq 2 N_T^0 L \|\mathbf{N}(t)-\hat{\mathbf{N}}(t)\|.
\end{eqnarray}

 Finally, we consider the death term

\begin{eqnarray*}
\|\hat{\mathbf{D}}({\mathbf{N}}(t))-\hat{\mathbf{D}}({\hat{\mathbf{N}}}(t))\|&&=\sum_{i=1}^{I}|\hat{D}_i(\mathbf{N}(t))-\hat{D}_i(\hat{\mathbf{N}}(t))|\\
&&\leq \sum_{i=1}^{I}\sum_{j=1}^{I}K(x_i,x_j)|N_i(t)N_j(t)-\hat{N}_i(t) \hat{N}_j(t)|\\
&&\leq L \sum_{i=1}^{I}\sum_{j=1}^{I}|N_i(t)N_j(t)-\hat{N}_i(t) \hat{N}_j(t)|.
\end{eqnarray*}

Again we use the same equality as before to get
\begin{equation}\label{lipschitz condition for D fpt}
 \|\hat{\mathbf{D}}({\mathbf{N}}(t))-\hat{\mathbf{D}}({\hat{\mathbf{N}}}(t))\|\leq 2 N_T^0 L\|\mathbf{N}(t)-\hat{\mathbf{N}}(t)\|.
\end{equation}
Therefore, the application of Theorem \ref{convergence theorem} implies the convergence of the fixed pivot technique for aggregation PBE (\ref{trunc pbe fpt}) and the convergence is of same order as the consistency.

\section{Numerical examples}\label{numerics fpt}
We now justify our mathematical results on the convergence by taking
a few numerical examples where we numerically evaluate the experimental order of
convergence (EOC). A detailed comparison of numerical
 results of number density and moments with analytical solutions can be found in \cite{Kumar:1996-1, J_Kumar:2006_Thesis}. All the numerical experiments are performed as in \cite{JKUMARNM:2008}.

First, we consider the following normally distributed initial
condition (NIC)
\begin{eqnarray*}
n(0, x) = \frac{1}{\sigma \sqrt{2\pi}} \exp\left[-\frac{(x-\mu)^2}{2\sigma^2}\right].
\end{eqnarray*}
In addition, we take the following aggregation sum and product kernels
\begin{equation}\label{sum kernel}
K(x,y)=k_0 (x+y) \ \ \mbox{and}\ \ K(x,y)=k_0 xy.
\end{equation}
Since analytical solutions are not available for the above initial
condition and aggregation kernels, we use the following formula in
order to calculate the experimental order of convergence

\begin{eqnarray*}
\mbox{EOC} = \ln\left(\frac{\|\hat{\mathbf{N}}_h - \hat{\mathbf{N}}_{h/2}\|}{\|\hat{\mathbf{N}}_{h/2} -
\hat{\mathbf{N}}_{h/4}\|}\right)\Big/\ln(2).
\end{eqnarray*}
Here $\hat{\mathbf{N}}_h$ represents the numerical solution on a
uniform mesh of width $h$. The other parameters are $\sigma^2 = 0.01$, $\mu = 1$
and $k_0 = 1$. Now we will calculate the EOC on five different types of uniform and non-uniform meshes.

Let us first calculate the EOC for uniform meshes. For a uniform mesh, we fix
 the minimum and maximum values of $x$ as $0$ and $15$,
respectively in the numerical computation. The number of grid points is denoted by
\textbf{GP} in the following tables. The numerical results are shown in Table \ref{T:Case1}. As expected
from the mathematical analysis, the numerical results show the convergence of second order.

Let us now consider the second case of non-uniform smooth meshes. In particular, we took the case of geometric
 grids which can be obtained by applying an exponential smooth transformation as $x = \exp(\xi)$. Here $\xi$ is the
variable with uniform grids. The computational domain in this
case is set as $[1e-6, 1000]$ which corresponds to the $\xi$
domain $[\ln(1e-6), \ln(1000)]$. The numerical results have been summarized in Table \ref{T:Case2}.
Once again the numerical results show that the fixed
 pivot technique gives second order convergence on non-uniform smooth meshes.

\begin{table}[h] \scriptsize
\begin{minipage}{50mm}
 \caption{Uniform grids (NIC)}
\begin{tabular}{@{}lll@{}}
\centering
\subtable[$K(x,y)=k_0 (x+y)$]{
 \centering
\begin{tabular}{@{}lll@{}}
\hline
 GP & Error $L_1$&\textbf{EOC}\\
\hline
 60 & - & -     \\
120 & 0.0598 & - \\
240 & 0.0178 & 1.75 \\
480 & 5.0E-3 & 1.82 \\
960 & 1.3E-3 & 1.95 \\
 \hline
\end{tabular}
}
\subtable[$K(x,y)=k_0 xy$]{
 \centering
\begin{tabular}{@{}lll@{}}
\hline
 Error $L_1$&\textbf{EOC}\\
\hline
 - & -     \\
 0.0306 & - \\
 8.4E-3 & 1.86 \\
 2.3E-3 & 1.89 \\
 6.0E-4 & 1.95 \\
 \hline
\end{tabular}
}
\end{tabular}
\label{T:Case1}
\end{minipage}
\hfil
\begin{minipage}{50mm}
\caption{Non-uniform smooth grids (NIC)}
\begin{tabular}{@{}lll@{}}
\centering
\subtable[$K(x,y)=k_0 (x+y)$]{
 \centering
\begin{tabular}{@{}lll@{}}
\hline
 GP & Error $L_1$&\textbf{EOC}\\
\hline
 60 & - & -     \\
120 & 0.0456 & - \\
240 & 0.0118 & 1.95 \\
480 & 3.0E-3 & 1.97 \\
960 & 7.6E-4 & 1.98 \\
 \hline
\end{tabular}
}
\subtable[$K(x,y)=k_0 xy$]{
 \centering
\begin{tabular}{@{}lll@{}}
\hline
 Error $L_1$&\textbf{EOC}\\
\hline
- & -     \\
 0.0374 & - \\
 9.6E-3 & 1.96 \\
 2.4E-3 & 2.00 \\
 6.0E-4 & 2.00 \\
 \hline
\end{tabular}
}
\end{tabular}
\label{T:Case2}
\end{minipage}
\end{table}

The third test case has been performed on a locally uniform mesh using the same computational
domain as is in the previous case. In this case we started the computation on $30$ geometric mesh points,
and then each cell was divided into two equal parts in the further refined levels of computation. In this way we obtained a locally uniform mesh. The
EOC has been summarized in Table \ref{T:Case3}. Table \ref{T:Case3} clearly shows that the fixed pivot technique is
only first order accurate.

Now we consider the fourth case of an oscillatory mesh to
evaluate the EOC. Let us take an example of oscillatory mesh, i.e.\
\vspace{-.20cm}
\begin{eqnarray*}
\Delta x_{i+1}:=
\left\{
    \begin{array}{ll}
        2\Delta x_{i} & \mbox{if } \ i\ \ \mbox{is odd}, \\
        \frac{1}{2}\Delta x_{i} & \mbox{if } \ i\ \ \mbox{is even}.
    \end{array}
\right.
\end{eqnarray*}
 Here the computational domain is the same as for the first case.
 First, we divide the computational domain in $30$
equidistant mesh points, and then each cell into two parts with $1:2$ as per further refined levels of computation.
 The numerical results has been shown in Table
\ref{T:Case4}. As expected, Table \ref{T:Case4} exhibits that the fixed pivot
technique is not convergent on oscillatory meshes.\\
\begin{table}[h] \scriptsize
\begin{minipage}{50mm}
\caption{Locally uniform grids (NIC)}
\begin{tabular}{@{}lll@{}}
\centering
\subtable[$K(x,y)=k_0 (x+y)$]{
 \centering
\begin{tabular}{@{}lll@{}}
\hline
 GP & Error $L_1$&\textbf{EOC}\\
\hline
 60 & - & -     \\
120 & 0.0416 & - \\
240 & 0.0212 & 0.97 \\
480 & 0.0105 & 1.01 \\
960 & 5.1E-3 & 1.04 \\
 \hline
\end{tabular}
}
\subtable[$K(x,y)=k_0 xy$]{
 \centering
\begin{tabular}{@{}lll@{}}
\hline
 Error $L_1$&\textbf{EOC}\\
\hline
  - & -     \\
 0.0254 & - \\
 0.0126 & 1.01 \\
 6.0E-3 & 1.07 \\
 2.8E-3 & 1.09 \\
 \hline
\end{tabular}
}
\end{tabular}
\label{T:Case3}
\end{minipage}
\hfil
\begin{minipage}{50mm}
\caption{Oscillatory grids (NIC)}
\begin{tabular}{@{}lll@{}}
\centering
\subtable[$K(x,y)=k_0 (x+y)$]{
 \centering
\begin{tabular}{@{}lll@{}}
\hline
 GP & Error $L_1$&\textbf{EOC}\\
\hline
60 & - & - \\
120 & 0.0650 & - \\
 240 & 0.0632 & 0.04 \\
 480 & 0.0523 & 0.27 \\
 960 & 0.0518 & 0.01 \\
 \hline
\end{tabular}
}
\subtable[$K(x,y)=k_0 xy$]{
 \centering
\begin{tabular}{@{}lll@{}}
\hline
 Error $L_1$&\textbf{EOC}\\
\hline
- & - \\
 0.0347 & - \\
  0.0309 & 0.16 \\
 0.0257 & 0.26 \\
  0.0253 & 0.01 \\
 \hline
\end{tabular}
}
\end{tabular}
\label{T:Case4}
\end{minipage}
\end{table}
Finally we consider the fifth case of a non-uniform random mesh. The
computations have been performed on the same domain as for the second case.
We started again with $30$ geometric mesh points,
  and then each cell was divided into two parts of random width in the further
   refined levels of computation. For each value of $I = 60, 120, 240, 480$,
   we performed five runs on different random grids and the relative $L_1$ errors
   were measured. The mean of these errors over five runs is used to calculate the EOC.
   The numerical results have been shown in Table \ref{T:Case5}. Table \ref{T:Case5} shows clearly that the fixed pivot technique is not convergent.

Next, we take an exponentially
decreasing initial condition (EIC), namely $n(0,x)=\exp(-\alpha x)$. The sum and product aggregation kernels in (\ref{sum kernel}) are again considered.
Since the analytical solution is known for the above initial conditions and kernels which can be found in \cite{Scott:1968, Aldous:1999}. Then the
experimental order of convergence can be determined by the formula $\mbox{EOC} = \ln(E_{I} / E_{2I}) / \ln(2)$,
where $E_{I}$ and $E_{2I}$ are the $L_1$ error norms. The subscripts $I$ and $2I$ correspond to the degrees of freedom.
We can calculate the error $E_{I}$ on a mesh with $I$ cells. The relative error has been calculated by dividing
the error $\|\mathbf{N}-\hat{\mathbf{N}}\|$ by $\|\mathbf{N}\|$. The parameter $\alpha=10$ is taken in the above initial condition. Again we computed the EOC on five different type of meshes as before.

In case of uniform mesh, we set the computational domain as [0,30] to evaluate the EOC numerically. The numerical results are summarized in Table \ref{T:Case11}. Again, we obtain the convergence of second order numerically.

\begin{table}[h]\scriptsize
\begin{minipage}{50mm}
\caption{Non-uniform random grids (NIC)}
\begin{tabular}{@{}lll@{}}
\centering
\subtable[$K(x,y)=k_0 (x+y)$]{
 \centering
\begin{tabular}{@{}lll@{}}
\hline
 GP & Error $L_1$&\textbf{EOC}\\
\hline
60 & - & -     \\
120 & 0.0229 & - \\
240 & 0.0375 & -0.71 \\
480 & 0.0406 & -0.11 \\
960 & 0.0402 & 0.01 \\
 \hline
\end{tabular}
}
\subtable[$K(x,y)=k_0 xy$]{
 \centering
\begin{tabular}{@{}lll@{}}
\hline
 Error $L_1$&\textbf{EOC}\\
\hline
 - & -  \\
 0.0112 & - \\
 9.4E-3 & 0.25 \\
 0.0135 & -0.51 \\
 0.0129 & 0.06 \\
 \hline
\end{tabular}
}
\end{tabular}
\label{T:Case5}
\end{minipage}
\hfil
\begin{minipage}{50mm}
\caption{Uniform grids (EIC)}
\begin{tabular}{@{}lll@{}}
\centering
\subtable[$K(x,y)=k_0 (x+y)$]{
 \centering
\begin{tabular}{@{}lll@{}}
\hline
 GP & Error $L_1$&\textbf{EOC}\\
\hline
60 & 0.0486 & - \\
120 & 0.0135 & 1.84 \\
240 & 3.5E-3 & 1.94 \\
480 & 9.0E-4 & 1.96 \\
 \hline
\end{tabular}
}
\subtable[$K(x,y)=k_0 xy$]{
 \centering
\begin{tabular}{@{}lll@{}}
\hline
 Error $L_1$&\textbf{EOC}\\
\hline
 0.0274 & - \\
 7.2E-3 & 1.92 \\
 1.9E-3 & 1.92 \\
 4.8E-4 & 1.98 \\
 \hline
\end{tabular}
}
\end{tabular}
\label{T:Case11}
\end{minipage}
\end{table}

Let us now evaluate the EOC on geometric grids which is a particular case of non-smooth grids. The numerical computations
 have been performed on the same computational domain as is for the case of geometric grids considered with the normal initial condition. The numerical results are presented in Table \ref{T:Case21} which shows once again the convergence of second order.

The EOC has been computed once
more on locally uniform, oscillatory and random meshes for the case mentioned above. The computational domain for locally uniform and random meshes is identical as for the previous case. However, we perform the computations on an oscillatory mesh using the same domain as is for the uniform mesh. The numerical result are demonstrated in Tables \ref{T:Case31}, \ref{T:Case41} and \ref{T:Case51}. These tables show that we acquire the convergence of first order on locally uniform mesh while the fixed pivot technique is zero order convergent on oscillatory and random meshes.\\
\begin{table}[h]\scriptsize
\begin{minipage}{50mm}
\caption{Non-uniform smooth grids (EIC)}
\begin{tabular}{@{}lll@{}}
\centering
\subtable[$K(x,y)=k_0 (x+y)$]{
 \centering
\begin{tabular}{@{}lll@{}}
\hline
 GP & Error $L_1$&\textbf{EOC}\\
\hline
 60 & 7.1E-3 & -     \\
120 & 1.8E-3 & 1.97 \\
 240 & 4.5E-4 & 2.00 \\
 480 & 1.1E-4 & 2.03 \\
 \hline
\end{tabular}
}
\subtable[$K(x,y)=k_0 xy$]{
 \centering
\begin{tabular}{@{}lll@{}}
\hline
 Error $L_1$&\textbf{EOC}\\
\hline
  6.3E-3 & -     \\
 1.6E-3 & 1.98 \\
 4.0E-4 & 2.00 \\
 1.0E-4 & 2.00 \\
 \hline
\end{tabular}
}
\end{tabular}
\label{T:Case21}
\end{minipage}
\hfil
\begin{minipage}{50mm}
\caption{Locally uniform grids (EIC)}
\begin{tabular}{@{}lll@{}}
\centering
\subtable[$K(x,y)=k_0 (x+y)$]{
 \centering
\begin{tabular}{@{}lll@{}}
\hline
 GP & Error $L_1$&\textbf{EOC}\\
\hline
60 & 0.0303 & - \\
 120 & 0.0156 & 0.96 \\
 240 & 7.7E-3 & 1.02 \\
 480 & 3.8E-3 & 1.03 \\
 \hline
\end{tabular}
}
\subtable[$K(x,y)=k_0 xy$]{
 \centering
\begin{tabular}{@{}lll@{}}
\hline
 Error $L_1$&\textbf{EOC}\\
\hline
 0.0145 & - \\
  7.1E-3 & 1.04 \\
 3.3E-3 & 1.08 \\
  1.6E-3 & 1.06 \\
 \hline
\end{tabular}
}
\end{tabular}
\label{T:Case31}
\end{minipage}
\end{table}
\vspace{-.50cm}
\begin{table}[h]\scriptsize
\begin{minipage}{50mm}
\caption{Oscillatory grids (EIC)}
\begin{tabular}{@{}lll@{}}
\centering
\subtable[$K(x,y)=k_0 (x+y)$]{
 \centering
\begin{tabular}{@{}lll@{}}
\hline
 GP & Error $L_1$&\textbf{EOC}\\
\hline
60 & 0.0554 & - \\
 120 & 0.0532 & 0.05 \\
 240 & 0.0539 & -0.01 \\
 480 & 0.0524 & 0.04 \\
 \hline
\end{tabular}
}
\subtable[$K(x,y)=k_0 xy$]{
 \centering
\begin{tabular}{@{}lll@{}}
\hline
 Error $L_1$&\textbf{EOC}\\
\hline
 0.0295 & - \\
  0.0298 & -0.01 \\
 0.0279 & 0.09 \\
  0.0256 & 0.12 \\
 \hline
\end{tabular}
}
\end{tabular}
\label{T:Case41}
\end{minipage}
\hfil
\begin{minipage}{50mm}
\caption{Non-uniform random grids (EIC)}
\begin{tabular}{@{}lll@{}}
\centering
\subtable[$K(x,y)=k_0 (x+y)$]{
 \centering
\begin{tabular}{@{}lll@{}}
\hline
 GP & Error $L_1$&\textbf{EOC}\\
\hline
60 & 0.0246 & - \\
 120 & 0.0292 & -0.25 \\
 240 & 0.0319 & -0.13 \\
 480 & 0.0380 & -0.25 \\
 \hline
\end{tabular}
}
\subtable[$K(x,y)=k_0 xy$]{
 \centering
\begin{tabular}{@{}lll@{}}
\hline
 Error $L_1$&\textbf{EOC}\\
\hline
 0.0162 & - \\
  0.0204 & -0.34 \\
 0.0222 & -0.12 \\
  0.0232 & -0.07 \\
 \hline
\end{tabular}
}
\end{tabular}
\label{T:Case51}
\end{minipage}
\end{table}

\section{Conclusions}\label{conclusion}
A detailed study on the convergence of the fixed pivot technique is given for solving aggregation PBE. It is ascertained that the technique is second order convergent on uniform and non-uniform smooth meshes. However, it is only first order convergent on locally uniform meshes. Finally, the scheme is examined closely on oscillatory and non-uniform random meshes and it is observed that the technique is not convergent. Furthermore, all observations are validated numerically.

\section*{Acknowledgments}
A.K. Giri would like to thank International Max-Planck Research
School, Magdeburg, Germany and FWF Austrian Science Fund
P21622-N18 for their support. We are also grateful to Prof. G.
Warnecke for his valuable discussions.

\end{document}